\documentclass[12pt]{article}
\usepackage{razb}

\newcommand{\refeq}[1]{{\rm (\ref{#1})}}
\newlength{\longeqmarginwidth}
\newlength{\longeqwidth}
\newlength{\longeqskiplength}
\newcommand{\longeq}[1]{
\settowidth{\longeqmarginwidth}{{\Large \{~\}}~(\theequation)}
\setlength{\longeqwidth}{\textwidth}
\addtolength{\longeqwidth}{-\longeqmarginwidth}
\left. \parbox{\longeqwidth}{\begin{eqnarray*} #1
\end{eqnarray*}} \right\}}

\newcommand{\Vzero}[2]{\mathord{\stackrel{\textstyle\kern1pt\circ}
{\smash V\vbox to6pt{}}}\vphantom{V}_{#1}^{#2}}

\newcommand{\function}[2]{:#1 \longrightarrow #2}
\newcommand{\of}[1]{\left( #1 \right)}

\newcommand{\set}[2]{\left\{\hspace{0.2ex} #1 \left|\: #2
\right. \right\}}

\newcommand{\eval}[2]{\llbracket #1 \rrbracket_{#2}}

\newcommand{\df}{\stackrel{\rm def}{=}}

\newcommand{\Hom}{{\rm Hom}}

\newcommand{\rn}{\boldsymbol}
\newcommand{\scr}{\mathcal}

\newcommand{\prob}[1]{ {\bf P}\! \left[ #1 \right] }
\newcommand{\expect}[1]{ {\bf E}\! \left[ #1 \right] }

\newcommand{\condexpect}[2]{ {\bf E}\! \left[ \hspace{0.5mm} #1
\left| \hspace{0.5mm} #2 \right. \right] }

\newcounter{operator}

\usepackage{amsfonts}
\usepackage{amsmath,amscd}
\usepackage{stmaryrd}
\usepackage{epic,eepic}

\newcommand{\FDF}{\text{FDF}}
\newcommand{\Graph}{\text{Graph}}
\newcommand{\C}{\text{C}}
\newcommand{\CE}{\text{CE}}

\title{On the Fon-der-Flaass Interpretation of Extremal Examples for Tur\'an's (3,4)-problem}
\author{Alexander A. Razborov\thanks{University of Chicago, {\tt razborov@cs.uchicago.edu}. Part of this work was done while the author was with
Steklov Mathematical Institute, supported by the Russian Foundation
for Basic Research, and with Toyota Technological Institute at Chicago.}}
\begin{document}
\maketitle

\begin{abstract}
In 1941, Tur\'an conjectured that the edge density of any 3-graph without independent sets on 4 vertices ({\em Tur\'an $(3,4)$-graph}) is $\geq 4/9(1-o(1))$, and he gave the first example witnessing this bound. Brown (1983) and Kostochka (1982) found many other examples of this density. Fon-der-Flaass (1988) presented a general construction that converts an arbitrary $\vec C_4$-free orgraph $\Gamma$ into a Tur\'an $(3,4)$-graph. He observed that all Tur\'an-Brown-Kostochka examples result from his construction, and proved the bound $\geq 3/7(1-o(1))$ on the edge density of any Tur\'an $(3,4)$-graph obtainable in this way.

\bigskip
In this paper we establish the optimal bound $4/9(1-o(1))$ on the edge density of any Tur\'an $(3,4)$-graph resulting from the Fon-der-Flaass construction under any of the following assumptions on the undirected graph $G$ underlying the orgraph $\Gamma$:
\begin{itemize}
\item $G$ is complete multipartite;

\item the edge density of $G$ is $\geq (2/3-\epsilon)$ for some absolute constant $\epsilon>0$.
\end{itemize}
We are also able to improve Fon-der-Flaass's bound to $7/16(1-o(1))$ without any extra assumptions on $\Gamma$.
\end{abstract}

\section{Introduction}

In the classical paper \cite{Man}, Mantel determined the minimal
number of edges a graph $G$ must have so that every three vertices
span at least one edge. In the paper
\cite{Tur} (that essentially started off the field of extremal
combinatorics), Tur\'an generalized Mantel's result to independent
sets of arbitrary size. He also asked if similar generalizations can
be obtained for hypergraphs, and these questions became notoriously
known ever since as one of the most difficult open problems in
discrete mathematics.

To be more specific, for a family $\scr H$ of $r$-uniform hypergraphs ($r$-graphs in what follows) and an integer $n$, let $\text{ex}_{\min}(n; \scr H)$ be the minimal possible number of edges  in an $n$-vertex $r$-graph not containing any of $H\in\scr H$ as an induced subgraph, and let

$$
\pi_{\min}(\scr H) \df \lim_{n\rightarrow\infty}
\frac{\text{ex}_{\min}(n; \scr H)}{{n\choose r}}
$$
(it is well-known that this limit exists).

For $\ell>r\geq 2$, let $I^r_\ell$ be the empty $r$-graph on $\ell$ vertices. Then we still do not know $\pi_{\min}(I_\ell^r)$ (not even to mention the exact value of $\text{ex}_{\min}(n; I^r_\ell)$) for {\em any} pair $\ell>r\geq 3$, although plausible conjectures do exist \cite{Sid}. The simplest unresolved case that has also received most attention is $r=3,\ \ell=4$; it is sometimes called {\em Tur\'an's $(3,4)$-problem} (see e.g. \cite{Kos}). Tur\'an himself exhibited an infinite family of 3-graphs witnessing $\pi_{\min}(I^3_4)\leq 4/9$ and conjectured that this is actually the right value. \cite{Dec}, Giraud (unpublished), \cite{GrL} proved
increasingly stronger lower bounds on $\pi_{\min}(I^3_4)$, with the
current record
$$
\pi_{\min}(I^3_4) \geq 0.438334
$$
\cite{turan}. The latter bound in fact represents the result of the numerical computation of a natural positive semi-definite program associated with the Tur\'an's $(3,4)$-problem; the best bound for which a ``human'' proof is available is
$$
\pi_{\min}(I^3_4)\geq \frac{9-\sqrt{17}}{12}\geq 0.406407
$$
\cite{GrL}. \cite{turan} also proved that $\pi_{\min}(I^3_4,G_3)= 4/9$, where $G_3$ is the 3-graph on 4 vertices with 3 edges. As an induced subgraph, $G_3$ is also missing in Tur\'an's original example, and Pikhurko \cite{Pik} proved that asymptotically this is the {\em only} example of a $\{I^3_4, G_3\}$-free 3-graph attaining the equality here.

\medskip
One prominent way to attack any extremal problem is by better understanding the structure of its (conjectured) extremal configurations. And one of the (many) difficulties associated with Tur\'an's $(3,4)$-problem is that this set is extraordinarily complex and consists of a continuous family of principally different configurations (in terminology of \cite{flag}, a family of different homomorphisms $\phi\in\Hom^+(\scr A^0,{\Bbb R})$ attaining $\phi(\rho_3)=4/9$) parameterized by probability distributions on the real line. These examples were constructed by Brown \cite{Bro} and Kostochka \cite{Kos}, and the only (to the best of our knowledge) successful attempt to ``explain'' them in ``external'' terms was undertaken by Fon-der-Flaass \cite{Fon}.

More specifically, let $\Gamma$ be an orgraph without induced (oriented) cycles $\vec C_4$. Fon-der-Flaass defined the 3-graph $FDF(\Gamma)$ with the same vertex sets by letting $(a,b,c)$ span an edge iff the induced orgraph $\Gamma|_{\{a,b,c\}}$ is isomorphic to one of $I_3,\bar P_3,\vec K_{1,2},\vec{T_3}$ on Figure \ref{m3}.
\begin{figure}[tb]
\begin{center}
\setlength{\unitlength}{0.254mm}
\begin{picture}(391,188)(13,-201)
        \special{color rgb 0 0 0}\allinethickness{0.254mm}\special{sh 0.99}\put(15,-65){\ellipse{4}{4}} 
        \special{color rgb 0 0 0}\allinethickness{0.254mm}\special{sh 0.99}\put(40,-15){\ellipse{4}{4}} 
        \special{color rgb 0 0 0}\allinethickness{0.254mm}\special{sh 0.99}\put(65,-65){\ellipse{4}{4}} 
        \special{color rgb 0 0 0}\allinethickness{0.254mm}\special{sh 0.99}\put(115,-65){\ellipse{4}{4}} 
        \special{color rgb 0 0 0}\allinethickness{0.254mm}\special{sh 0.99}\put(140,-15){\ellipse{4}{4}} 
        \special{color rgb 0 0 0}\allinethickness{0.254mm}\special{sh 0.99}\put(165,-65){\ellipse{4}{4}} 
        \special{color rgb 0 0 0}\put(35,-91){\shortstack{$I_3$}} 
        \special{color rgb 0 0 0}\put(135,-91){\shortstack{$\bar P_3$}} 
        \special{color rgb 0 0 0}\allinethickness{0.254mm}\special{sh 0.99}\put(215,-65){\ellipse{4}{4}} 
        \special{color rgb 0 0 0}\allinethickness{0.254mm}\special{sh 0.99}\put(240,-15){\ellipse{4}{4}} 
        \special{color rgb 0 0 0}\allinethickness{0.254mm}\special{sh 0.99}\put(265,-65){\ellipse{4}{4}} 
        \special{color rgb 0 0 0}\allinethickness{0.254mm}\special{sh 0.99}\put(315,-65){\ellipse{4}{4}} 
        \special{color rgb 0 0 0}\allinethickness{0.254mm}\special{sh 0.99}\put(340,-15){\ellipse{4}{4}} 
        \special{color rgb 0 0 0}\allinethickness{0.254mm}\special{sh 0.99}\put(365,-65){\ellipse{4}{4}} 
        \special{color rgb 0 0 0}\put(230,-91){\shortstack{$\vec K_{1,2}$}} 
        \special{color rgb 0 0 0}\put(330,-91){\shortstack{$\vec K_{2,1}$}} 
        \special{color rgb 0 0 0}\allinethickness{0.254mm}\special{sh 0.99}\put(65,-175){\ellipse{4}{4}} 
        \special{color rgb 0 0 0}\allinethickness{0.254mm}\special{sh 0.99}\put(90,-125){\ellipse{4}{4}} 
        \special{color rgb 0 0 0}\allinethickness{0.254mm}\special{sh 0.99}\put(115,-175){\ellipse{4}{4}} 
        \special{color rgb 0 0 0}\allinethickness{0.254mm}\special{sh 0.99}\put(165,-175){\ellipse{4}{4}} 
        \special{color rgb 0 0 0}\allinethickness{0.254mm}\special{sh 0.99}\put(190,-125){\ellipse{4}{4}} 
        \special{color rgb 0 0 0}\allinethickness{0.254mm}\special{sh 0.99}\put(215,-175){\ellipse{4}{4}} 
        \special{color rgb 0 0 0}\put(85,-201){\shortstack{$\vec P_3$}} 
        \special{color rgb 0 0 0}\put(185,-201){\shortstack{$\vec C_3$}} 
        \special{color rgb 0 0 0}\allinethickness{0.254mm}\special{sh 0.99}\put(265,-175){\ellipse{4}{4}} 
        \special{color rgb 0 0 0}\allinethickness{0.254mm}\special{sh 0.99}\put(290,-125){\ellipse{4}{4}} 
        \special{color rgb 0 0 0}\allinethickness{0.254mm}\special{sh 0.99}\put(315,-175){\ellipse{4}{4}} 
        \special{color rgb 0 0 0}\put(285,-201){\shortstack{$\vec T_3$}} 
        \special{color rgb 0 0 0}\allinethickness{0.254mm}\put(115,-65){\vector(1,0){50}} 
        \special{color rgb 0 0 0}\allinethickness{0.254mm}\put(265,-175){\vector(1,0){50}} 
        \special{color rgb 0 0 0}\allinethickness{0.254mm}\put(240,-15){\vector(-1,-2){25}} 
        \special{color rgb 0 0 0}\allinethickness{0.254mm}\put(240,-15){\vector(1,-2){25}} 
        \special{color rgb 0 0 0}\allinethickness{0.254mm}\put(90,-125){\vector(1,-2){25}} 
        \special{color rgb 0 0 0}\allinethickness{0.254mm}\put(190,-125){\vector(1,-2){25}} 
        \special{color rgb 0 0 0}\allinethickness{0.254mm}\put(290,-125){\vector(1,-2){25}} 
        \special{color rgb 0 0 0}\allinethickness{0.254mm}\put(315,-65){\vector(1,2){25}} 
        \special{color rgb 0 0 0}\allinethickness{0.254mm}\put(65,-175){\vector(1,2){25}} 
        \special{color rgb 0 0 0}\allinethickness{0.254mm}\put(165,-175){\vector(1,2){25}} 
        \special{color rgb 0 0 0}\allinethickness{0.254mm}\put(265,-175){\vector(1,2){25}} 
        \special{color rgb 0 0 0}\allinethickness{0.254mm}\put(365,-65){\vector(-1,2){25}} 
        \special{color rgb 0 0 0}\allinethickness{0.254mm}\put(215,-175){\vector(-1,0){50}} 
        \special{color rgb 0 0 0} 
\end{picture}
\caption{\label{m3} Simple orgraphs on 3 vertices}
\end{center}
\end{figure}
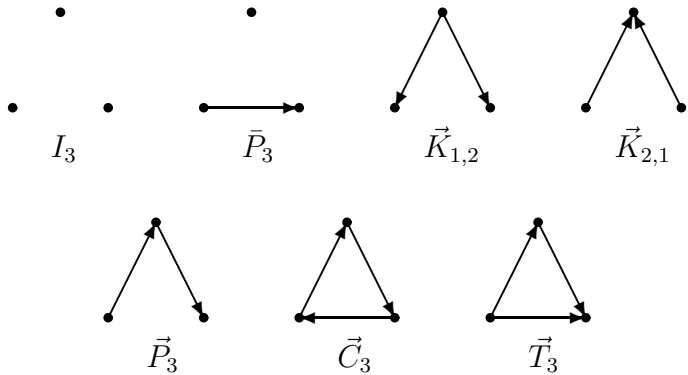
In his brief note, Fon-der-Flaass proved the following remarkable facts:
\begin{enumerate}
\item for any $\vec C_4$-free orgraph $\Gamma$, the 3-graph $FDF(\Gamma)$ is $I^3_4$-free;

\item all known extremal configurations \cite{Tur,Bro,Kos} have the form $FDF(\Gamma)$, where $\Gamma$ runs through a special class of orientations of the almost balanced complete 3-partite graph;

\item \label{c} for any $\vec C_4$-free $\Gamma$, the edge density of $FDF(\Gamma)$ is $\geq 3/7(1-o(1))$.
\end{enumerate}
The question of improving the latter bound to its conjectured optimal value $\geq 4/9(1-o(1))$, that we refer to as {\em Fon-der-Flaass conjecture}\footnote{in \cite{Fon} it was actually asked in more elusive form {\tt it remains unclear if this gap suffices for constructing in this way a counterexample to Tur\'an's conjecture} }, is still open, and this is precisely the question that motivated our paper.

While we have not been able to solve it completely, our main result (Theorem \ref{main}) achieves this goal under any one of the following two assumptions:
\begin{enumerate}
\item $\Gamma$ is an orientation of a complete multipartite graph (possibly unbalanced and with more than three parts);

\item the edge density of $\Gamma$ is $\geq 2/3-\epsilon$, where $\epsilon>0$ is an absolute constant.
\end{enumerate}
Note that both these assumptions are fulfilled by the orgraphs underlying Tur\'an-Brown-Kostochka examples, an, therefore, both these results can be interpreted as proving the desired bound $4/9(1-o(1))$ for classes of $I^3_4$-free 3-graphs containing all known conjectured extremal configurations and having a ``reasonably invariant'' description. To the best of our knowledge, this is the first result of this kind for Tur\'an's $(3,4)$-problem.

Remarkably, our second restriction (on the edge density) is precisely the same as the one used by Lov\'asz and Simonovits in their early work on the density of triangles in graphs \cite{LoSi} (for further developments see \cite{Fish,triangles,Nik}), and this bound is used as one of our starting points (see \refeq{main_1}). But we have not been able to make more formal connection between the two results and techniques used in their proofs.

Also, our second result readily implies that, provided there indeed are no extremal configurations with edge density asymptotically different from 2/3, Fone-der-Flaass conjecture can be in principle proved by brute-force using methods similar to those from \cite{turan}: as the number of vertices $\ell$ increases, the lower bound on the edge density for which the semi-definite program gives the desired result, must converge to 2/3. At the moment, however, this observation is of purely theoretical value: the best one can achieve with brute force on 5 vertices (when we already have 559 $\vec C_4$-free orgraphs) is to prove the desired bound $4/9(1-o(1))$ when the edge density of $\Gamma$ is $\leq 2/3-1.3\cdot 10^{-2}$. $1.3\cdot 10^{-2}$ is, however, {\em way} above any lower bound on $\epsilon$ in Theorem \ref{main} \ref{main:b} one can realistically hope to achieve with our methods.

\bigskip
Our proofs rely on a very substantial amount of (double) counting, and they are presented in the same framework of Flag Algebras they were found in. As an additional benefit of this approach, we get ``stable'' versions of our results for free (that is, without using hypergraph regularity lemmas): all of them hold even if we only know that the density of induced $\vec C_4$ (and the density of $\bar P_3$ in the part regarding complete multipartite graphs) is $o(1)$.  At this point, we assume certain familiarity with \cite{flag} and instead of trying to duplicate basic definitions, simply give pointers to relevant places in \cite{flag} as we go along. We also try to provide as much combinatorial intuition, examples and informal explanations as possible. In this way we are in particular able to include a few interesting observations about Tur\'an-Brown-Kostochka examples that hardly qualify as independent theorems but nonetheless might turn out to be useful for future research on Tur\'an's (3,4)-problem.

\medskip
The paper is organized as follows. In Section \ref{prel} we remind some necessary facts (recast in our language) that pertain to the previous work on Tur\'an's $(3,4)$-problem and state our main results. This is interlaced with introducing some new general notation pertaining to Flag Algebras (Section \ref{addons}), as well as with providing a registry of all concrete theories, types, flags etc. needed for our work (Section \ref{names}). In Section \ref{warm-up} we give a relatively easy improvement of Fon-der-Flaass's bound (without any extra assumptions on $\Gamma$) to $7/16(1-o(1))$, at the same time introducing some prerequisites for proving our main result, Theorem \ref{main}. This theorem itself is proved in Section \ref{main:sec}. We conclude with a few remarks in Section \ref{concl}.

\section{Preliminaries} \label{prel}

A 3-uniform hypergraph will be called a {\em 3-graph}. Given a 3-graph $H$ and an integer $n$, let $\text{ex}_{\min}(n; H)$ be the minimal possible number of edges in an $n$-vertex $3$-graph not containing $H$ as an induced subgraph, and let
$$
\pi_{\min}(H) \df \lim_{n\rightarrow\infty}\frac{\text{ex}_{\min}(n; H)}{{n\choose 3}}
$$
(it is well-known that this limit exists). In this paper we are interested in $\pi_{\min}(I^3_4)$, where $I^3_4$ is the empty 3-graph on 4 vertices. We call a 3-graph {\em Tur\'an} if it does not contain copies of $I^3_4$.
An oriented graph, or an {\em orgraph} is a directed graph without loops, multiple edges and pairs of edges of the form $\langle v,w\rangle, \langle w,v\rangle$. For an orgraph $\Gamma$ and $V\subseteq V(\Gamma)$, $\Gamma|_V$ is the induced subgraph spanned by vertices in $V$. $\vec C_4$ is an oriented cycle on 4 vertices, and an orgraph $\Gamma$ is {\em $\vec C_4$-free} if it does not contain {\em induced} copies of $\vec C_4$.

\smallskip
This whole paper is motivated by the famous {\em Tur\'an's $(3,4)$-problem}:
\begin{center}
\tt Determine $\pi_{\min}(I^3_4)$.
\end{center}
Let us first review, in appropriate terms, the set of conjectured extremal configurations \cite{Tur,Bro,Kos,Fon}.

Let $\Omega={\Bbb Z_3}\times {\Bbb R}$, and consider the (infinite) orgraph $\Gamma_K=(\Omega,E_K)$ on $\Omega$ given by
$$
E_K\df \set{\langle (a,x),(b,y)\rangle}{(x+y<0 \land b=a+1) \lor (x+y>0 \land b=a-1)}.
$$

\begin{fact} \label{fact1}
$\Gamma_K$ is $\vec C_4$-free.
\end{fact}
\begin{proof}
Any $\vec C_4$-cycle in $\Gamma_K$ must clearly be of the form $ (a,x_1),(a+1,y_1),(a,x_2),(a+1,y_2)$ for some $a\in {\Bbb Z}_3$. But then $\langle (a,x_1),(a+1,y_1)\rangle \in E_K$ and $\langle (a,x_2),(a+1,y_2)\rangle \in E_K$ imply $x_1+y_1+x_2+y_2=(x_1+y_1)+(x_2+y_2)<0$, and, by considering the other two edges, $x_1+y_1+x_2+y_2>0$. Contradiction.
\end{proof}

Let now $\Gamma=(V,E)$ be an arbitrary (possibly infinite) orgraph. We define $FDF(\Gamma)$ as the 3-graph on $V$ in which $(u,v,w)$ spans an edge if and only if $\Gamma|_{\{u,v,w\}}$ either contains an isolated vertex of both in-degree and out-degree 0, or it contains a vertex of out-degree 2.

\begin{fact}[\cite{Fon}]
If $\Gamma$ is an arbitrary $\vec C_4$-free orgraph then the 3-graph $FDF(\Gamma)$ is Tur\'an.
\end{fact}

Let now $S\subseteq {\Bbb R}$ be a finite subset such that $x+y\neq 0$ for any $x,y\in S$.

\begin{fact}[\cite{Fon}] \label{fact3}
The edge density of $FDF\of{\Gamma_K|_{{\Bbb Z}_3\times S}}$ tends to $4/9$ as $|S|\rightarrow\infty$.
\end{fact}

Facts \ref{fact1}-\ref{fact3} together imply that $\pi_{\min}(I^3_4)\leq 4/9$. Tur\'an's original construction \cite{Tur} corresponds to the case when all elements in $S$ are positive. Brown's examples \cite{Bro} are obtained when negative entries in $S$ are allowed, but are always smaller in absolute value than all positive entries. Kostochka's examples \cite{Kos} correspond to arbitrary sets $S$. While tiny variations of these examples are known in terms of conjectured exact values $\text{ex}_{\min}(n; I^3_4)$, asymptotically Tur\'an 3-graphs of the form $FDF(\Gamma_K|{{\Bbb Z}_3\times S})$ constitute the complete set of known extremal configurations. Statements of this sort are best formulated and viewed in a ``limit'' framework, and we defray further discussion until we review some material from Flag Algebras.

\subsection{Flag Algebras}

As we stated in Introduction, we do not attempt to give here a self-contained account of the general theory developed in \cite{flag}. The sole purpose of this section is to introduce a little bit more of handy notation {\em in addition} to \cite{flag} and fix names (many of them are listed on the corresponding pictures) for concrete theories, interpretations, types, flags etc. used in the rest of the paper.

\subsubsection{Add-ons} \label{addons}

We will call an open interpretation $(U,I):T_1^{\sigma_1}\leadsto T_2^{\sigma_2}$ (\cite[Definition 4]{flag}) {\em total} if the relativizing predicate $U$ is trivial, that is $U(x)\equiv\top$. All interpretations considered in this paper are total, and we omit $U$ from all definitions and notation. A total interpretation $I$ is {\em global} if additionally $\sigma_1=\sigma_2=0$ are trivial types (global interpretations were considered in \cite[\S 2.3.3]{flag}). For a global interpretation $I:T_1\leadsto T_2$ and a type $\sigma$ of the theory $T_2$, we have the induced interpretation $I^\sigma:T_2^{I(\sigma)}\leadsto T_1^\sigma$.

Any total interpretation $I:T_1^{\sigma_1}\leadsto T_2^{\sigma_2}$ defines an algebra homomorphism $\pi^I\function{\scr A^{\sigma_1}[T_1]}{\scr A^{\sigma_2}[T_2]}$ (\cite[Theorem 2.6]{flag}). If $I$ is global and $\sigma$ is a type of $T_2$, $\pi^{(I^\sigma)}$ will be abbreviated to $\pi^{I,\sigma}$. For any total interpretation $I:T_1^{\sigma_1}\leadsto T_2^{\sigma_2}$ and $\phi\in\Hom^+(\scr A^{\sigma_2}[T_2], {\Bbb R})$, we abbreviate $\phi\pi^I\in \Hom^+(\scr A^{\sigma_1}[T_1], {\Bbb R})$ to $\phi|_I$. Likewise, for $\phi\in\Hom^+(\scr A^\sigma,{\Bbb R})$ and an injective mapping $\eta\function{[k']}{[|\sigma|]}$, $\phi\pi^{\sigma,\eta}\in\Hom^+(\scr A^\sigma, {\Bbb R})$ (\cite[\S 2.3.1]{flag}) will be abbreviated to $\phi|_\eta$. $\eta$ itself will be normally written down as its table $[\eta(1),\ldots,\eta(k')]$ and, moreover, when $k'=1$, the brackets will be omitted. For example, for $\phi\in\Hom^+(\scr A^\sigma, {\Bbb R})$, $\phi|_2$ means the composed homomorphism $\phi\pi^{\sigma,2}\df\phi\pi^{\sigma,\eta}$, where $\eta\function{[1]}{[|\sigma|]};\ \eta(1)=2$.

\smallskip
Recall \cite[Definition 6]{flag} that $f\leq_\sigma g$ for $f,g\in\scr A^\sigma$ means that $\phi(f)\leq \phi(g)$ for any $\phi\in\Hom^+(\scr A^\sigma,{\Bbb R})$. We naturally extend this notation to the case of arbitrary functions $g_1,\ldots,g_\ell,h_1,\ldots,h_\ell\function{{\Bbb R}^t}{{\Bbb R}}$, an arbitrary propositional formula $A(p_1,\ldots,p_\ell)$ and elements $f_1,\ldots,f_t\in\scr A^\sigma$ as follows: $A(g_1(f_1,\ldots,f_t)\geq h_1(f_1,\ldots,f_t),\ldots,g_\ell(f_1,\ldots,f_t)\geq h_\ell(f_1,\ldots,f_t))$ by definition means that $A(g_1(\phi(f_1),\ldots,\phi(f_t))\geq h_1(\phi(f_1),\ldots,\phi(f_t)),\ldots,g_\ell(\phi(f_1),\ldots,\phi(f_t))\geq h_\ell(\phi(f_1),\ldots,\phi(f_t)))$ is true for any $\phi\in\Hom^+(\scr A^\sigma,{\Bbb R})$.

\smallskip
For a non-degenerate type $\sigma_0$ of size $k_0$, $\phi\in\Hom^+(\scr A^{\sigma_0},{\Bbb R})$ and an extension $(\sigma,\eta)$ of $\sigma_0$ with $\phi_0((\sigma,\eta))>0$, we well denote by $S^{\sigma,\eta}(\phi)$ the support of the probability measure ${\bf P}^{\sigma,\eta}$ \cite[Definition 8]{flag}. This is the minimal closed subset of $\Hom^+(\scr A^\sigma,{\Bbb R})$ with the property $\prob{\rn{\phi^{\sigma,\eta}}\in S^{\sigma,\eta}(\phi)}=1$, and, as always, when $\sigma_0=0$, $\eta$ will be dropped from notation.

\subsubsection{Names} \label{names}

In this section, we fix notation for almost all concrete objects used in our paper.

\begin{center}
{\bf Theories.}
\end{center}

$T_{\text{Turan}}$ -- the theory of 3-graphs without copies of $I^3_4$.

$T_{\FDF}$ -- the theory of $\vec C_4$-free orgraphs.

$T_{\Graph}$ -- the theory of ordinary (simple, undirected) graphs.

For a theory $T$, we denote by $T^\ast$ its extension obtained by introducing three new unary predicate symbols $\chi^0,\chi^1,\chi^2$, together with the axioms
$$
\bigvee_{a\in {\Bbb Z}_3}\chi^a(v),\ \ \ \neg(\chi^a(v)\land \chi^b(v))\ (a\neq b\in {\Bbb Z_3}).
$$
In other words, we augment the theory $T$ with a ${\Bbb Z}_3$-coloring $\chi$ of the vertices not a priori related to the original structure provided by $T$.

\begin{center}
{\bf Interpretations.}
\end{center}
For the general set-up see \cite[Section 2.3]{flag}.

$\FDF:T_{\text{Turan}}\leadsto T_{\FDF}$ is the {\em Fon-der-Flaass interpretation}. Denoting the (only) predicate of the theory $T_{\text{Turan}}$ by $E_3$, and the predicate of $T_{\FDF}$ by $E$, we set
\begin{eqnarray*}
&&\FDF(E_3)(v_0,v_1,v_2) \df \bigwedge_{a\neq b\in {\Bbb Z}_3}(v_a\neq v_b) \land \left(\bigvee_{a\in {\Bbb Z}_3}(E(v_a,v_{a+1})\land E(v_a,v_{a-1}))\right.\\ &&\hspace{\parindent} \lor \left.\bigvee _{a\in {\Bbb Z}_3}(\neg E(v_a,v_{a+1})\land \neg E(v_a,v_{a-1})\land \neg E(v_{a-1},v_a)\land \neg E(v_{a+1},v_a))\right).
\end{eqnarray*}

{\em Orientation-erasing interpretation} $\text{OE}:T_{\Graph} \leadsto T_{\FDF}$ is defined in the obvious way. When working in the theory $T_{\FDF}$, we will systematically omit $\pi^{\text{OE}}$ and, for $f\in\scr A^0[T_{\Graph}]$, abbreviate $\pi^{\text{OE}}(f)$ to $f$.

{\em Color-erasing interpretation} $\CE_T:T\leadsto T^\ast$ is also obvious. The subscript $T$ is omitted whenever it is clear from the context.

The generalization $\text{OE}^\ast:T_{\Graph}^\ast \leadsto T_{\FDF}^\ast$ is also straightforward (this interpretation does not affect the ${\Bbb Z}_3$-coloring). Again, given $f\in\scr A^0[T_{\Graph}^\ast]$, we will abbreviate the element $\pi^{\text{OE}^\ast}\pi^{CE_{T_{\Graph}}}(f) = \pi^{\text{CE}_{T_\FDF}}\pi^{\text{OE}}(f)$ to $f$.

All these interpretations are global. We will define one crucial non-global interpretation a little bit later, after introducing some more notation.

\begin{center}
{\bf Models, Types and Flags.}
\end{center}
For any theory $T$, 0 is the trivial type of size 0. In vertex-uniform (\cite[Definition 11]{flag}) theories $T_{\text{Turan}},T_{\FDF},T_{\Graph}$, 1 is the only type of size 1.

\begin{center}
$T_{\text{Turan}}$
\end{center}
$\rho_3\in\scr M_3[T_{\text{Turan}}]$ is an edge.

\begin{center}
$T_{\Graph}$
\end{center}

$\scr M_2[T_{\Graph}]$ consists of two elements: $\rho$ (an edge) and $\nu$ (non-edge).

$|\scr M_3[T_{\Graph}]|=4$. Its elements, along with their adopted names, are listed on Figure \ref{graph3}.

\begin{figure}[tb]
\begin{center}
\setlength{\unitlength}{0.254mm}
\begin{picture}(354,78)(13,-91)
        \special{color rgb 0 0 0}\allinethickness{0.254mm}\special{sh 0.99}\put(15,-65){\ellipse{4}{4}} 
        \special{color rgb 0 0 0}\allinethickness{0.254mm}\special{sh 0.99}\put(40,-15){\ellipse{4}{4}} 
        \special{color rgb 0 0 0}\allinethickness{0.254mm}\special{sh 0.99}\put(65,-65){\ellipse{4}{4}} 
        \special{color rgb 0 0 0}\allinethickness{0.254mm}\special{sh 0.99}\put(115,-65){\ellipse{4}{4}} 
        \special{color rgb 0 0 0}\allinethickness{0.254mm}\special{sh 0.99}\put(140,-15){\ellipse{4}{4}} 
        \special{color rgb 0 0 0}\allinethickness{0.254mm}\special{sh 0.99}\put(165,-65){\ellipse{4}{4}} 
        \special{color rgb 0 0 0}\put(30,-91){\shortstack{$I_3$}} 
        \special{color rgb 0 0 0}\put(130,-91){\shortstack{$\bar P_3$}} 
        \special{color rgb 0 0 0}\allinethickness{0.254mm}\special{sh 0.99}\put(215,-65){\ellipse{4}{4}} 
        \special{color rgb 0 0 0}\allinethickness{0.254mm}\special{sh 0.99}\put(240,-15){\ellipse{4}{4}} 
        \special{color rgb 0 0 0}\allinethickness{0.254mm}\special{sh 0.99}\put(265,-65){\ellipse{4}{4}} 
        \special{color rgb 0 0 0}\allinethickness{0.254mm}\special{sh 0.99}\put(315,-65){\ellipse{4}{4}} 
        \special{color rgb 0 0 0}\allinethickness{0.254mm}\special{sh 0.99}\put(340,-15){\ellipse{4}{4}} 
        \special{color rgb 0 0 0}\allinethickness{0.254mm}\special{sh 0.99}\put(365,-65){\ellipse{4}{4}} 
        \special{color rgb 0 0 0}\put(230,-91){\shortstack{$P_3$}} 
        \special{color rgb 0 0 0}\allinethickness{0.254mm}\path(215,-65)(240,-15) 
        \special{color rgb 0 0 0}\allinethickness{0.254mm}\path(240,-15)(265,-65) 
        \special{color rgb 0 0 0}\allinethickness{0.254mm}\path(315,-65)(340,-15) 
        \special{color rgb 0 0 0}\allinethickness{0.254mm}\path(340,-15)(365,-65) 
        \special{color rgb 0 0 0}\put(330,-91){\shortstack{$K_3$}} 
        \special{color rgb 0 0 0}\allinethickness{0.254mm}\path(315,-65)(365,-65) 
        \special{color rgb 0 0 0}\allinethickness{0.254mm}\path(115,-65)(165,-65) 
        \special{color rgb 0 0 0} 
\end{picture}
\caption{\label{graph3} Ordinary graphs on 3 vertices}
\end{center}
\end{figure}
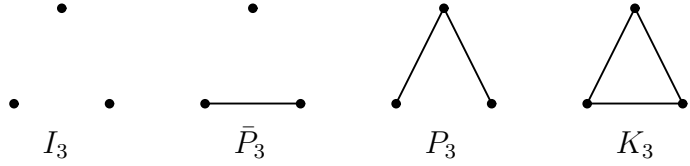

$E$ is the type of size 2 that is an edge, and $N$ is the type of size 2 that is non-edge.

$e\in\scr F^1_2[T_{\Graph}]$ is an edge with one distinguished vertex. There are two ways to turn $\bar P_3$ into an 1-flag: $\bar P_3^{1,c}$ (by selecting the isolated vertex) and $\bar P_3^{1,b}$ (by selecting any other). $\scr F_3^E[T_{\Graph}]$, $\scr F_3^N[T_{\Graph}]$ consist of four elements each; they are depicted on Figures \ref{e3}, \ref{n3} respectively.

\begin{figure}[tb]
\begin{center}
\setlength{\unitlength}{0.254mm}
\begin{picture}(443,83)(35,-111)
        \special{color rgb 0 0 0}\allinethickness{0.254mm}\special{sh 0.2}\put(80,-80){\ellipse{70}{20}} 
        \special{color rgb 0 0 0}\allinethickness{0.254mm}\special{sh 0.2}\put(195,-80){\ellipse{70}{20}} 
        \special{color rgb 0 0 0}\allinethickness{0.254mm}\special{sh 0.99}\put(55,-80){\ellipse{4}{4}} 
        \special{color rgb 0 0 0}\allinethickness{0.254mm}\special{sh 0.99}\put(80,-30){\ellipse{4}{4}} 
        \special{color rgb 0 0 0}\allinethickness{0.254mm}\special{sh 0.99}\put(105,-80){\ellipse{4}{4}} 
        \special{color rgb 0 0 0}\allinethickness{0.254mm}\special{sh 0.99}\put(170,-80){\ellipse{4}{4}} 
        \special{color rgb 0 0 0}\allinethickness{0.254mm}\special{sh 0.99}\put(195,-30){\ellipse{4}{4}} 
        \special{color rgb 0 0 0}\allinethickness{0.254mm}\special{sh 0.99}\put(220,-80){\ellipse{4}{4}} 
        \special{color rgb 0 0 0}\put(35,-86){\shortstack{$1$}} 
        \special{color rgb 0 0 0}\put(150,-86){\shortstack{$1$}} 
        \special{color rgb 0 0 0}\put(120,-86){\shortstack{$2$}} 
        \special{color rgb 0 0 0}\put(235,-86){\shortstack{$2$}} 
        \special{color rgb 0 0 0}\allinethickness{0.254mm}\path(105,-80)(55,-80) 
        \special{color rgb 0 0 0}\allinethickness{0.254mm}\path(195,-30)(220,-80) 
        \special{color rgb 0 0 0}\allinethickness{0.254mm}\path(220,-80)(170,-80) 
        \special{color rgb 0 0 0}\put(185,-111){\shortstack{$P_3^{E,b}$}} 
        \special{color rgb 0 0 0}\put(70,-111){\shortstack{$\bar P_3^E$}} 
        \special{color rgb 0 0 0}\allinethickness{0.254mm}\special{sh 0.2}\put(305,-80){\ellipse{70}{20}} 
        \special{color rgb 0 0 0}\allinethickness{0.254mm}\special{sh 0.2}\put(420,-80){\ellipse{70}{20}} 
        \special{color rgb 0 0 0}\allinethickness{0.254mm}\special{sh 0.99}\put(280,-80){\ellipse{4}{4}} 
        \special{color rgb 0 0 0}\allinethickness{0.254mm}\special{sh 0.99}\put(305,-30){\ellipse{4}{4}} 
        \special{color rgb 0 0 0}\allinethickness{0.254mm}\special{sh 0.99}\put(330,-80){\ellipse{4}{4}} 
        \special{color rgb 0 0 0}\allinethickness{0.254mm}\special{sh 0.99}\put(395,-80){\ellipse{4}{4}} 
        \special{color rgb 0 0 0}\allinethickness{0.254mm}\special{sh 0.99}\put(420,-30){\ellipse{4}{4}} 
        \special{color rgb 0 0 0}\allinethickness{0.254mm}\special{sh 0.99}\put(445,-80){\ellipse{4}{4}} 
        \special{color rgb 0 0 0}\put(260,-86){\shortstack{$1$}} 
        \special{color rgb 0 0 0}\put(375,-86){\shortstack{$1$}} 
        \special{color rgb 0 0 0}\put(345,-86){\shortstack{$2$}} 
        \special{color rgb 0 0 0}\put(460,-86){\shortstack{$2$}} 
        \special{color rgb 0 0 0}\allinethickness{0.254mm}\path(280,-80)(305,-30) 
        \special{color rgb 0 0 0}\allinethickness{0.254mm}\path(330,-80)(280,-80) 
        \special{color rgb 0 0 0}\allinethickness{0.254mm}\path(420,-30)(445,-80) 
        \special{color rgb 0 0 0}\allinethickness{0.254mm}\path(445,-80)(395,-80) 
        \special{color rgb 0 0 0}\put(410,-111){\shortstack{$K_3$}} 
        \special{color rgb 0 0 0}\put(295,-111){\shortstack{$P_3^{E,c}$}} 
        \special{color rgb 0 0 0}\allinethickness{0.254mm}\path(395,-80)(420,-30) 
        \special{color rgb 0 0 0} 
\end{picture}
\caption{\label{e3} $E$-flags on 3 vertices}
\end{center}
\end{figure}
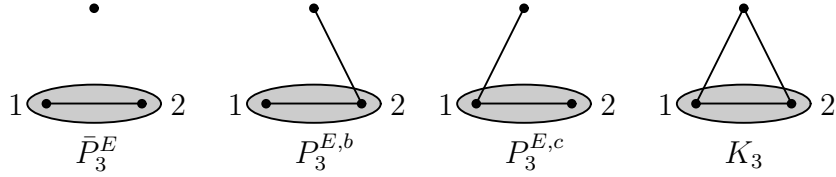

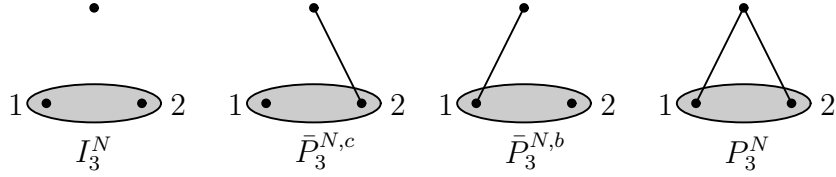
\begin{figure}[tb]
\begin{center}
\setlength{\unitlength}{0.254mm}
\begin{picture}(443,83)(35,-111)
        \special{color rgb 0 0 0}\allinethickness{0.254mm}\special{sh 0.2}\put(80,-80){\ellipse{70}{20}} 
        \special{color rgb 0 0 0}\allinethickness{0.254mm}\special{sh 0.2}\put(195,-80){\ellipse{70}{20}} 
        \special{color rgb 0 0 0}\allinethickness{0.254mm}\special{sh 0.99}\put(55,-80){\ellipse{4}{4}} 
        \special{color rgb 0 0 0}\allinethickness{0.254mm}\special{sh 0.99}\put(80,-30){\ellipse{4}{4}} 
        \special{color rgb 0 0 0}\allinethickness{0.254mm}\special{sh 0.99}\put(105,-80){\ellipse{4}{4}} 
        \special{color rgb 0 0 0}\allinethickness{0.254mm}\special{sh 0.99}\put(170,-80){\ellipse{4}{4}} 
        \special{color rgb 0 0 0}\allinethickness{0.254mm}\special{sh 0.99}\put(195,-30){\ellipse{4}{4}} 
        \special{color rgb 0 0 0}\allinethickness{0.254mm}\special{sh 0.99}\put(220,-80){\ellipse{4}{4}} 
        \special{color rgb 0 0 0}\put(35,-86){\shortstack{$1$}} 
        \special{color rgb 0 0 0}\put(150,-86){\shortstack{$1$}} 
        \special{color rgb 0 0 0}\put(120,-86){\shortstack{$2$}} 
        \special{color rgb 0 0 0}\put(235,-86){\shortstack{$2$}} 
        \special{color rgb 0 0 0}\allinethickness{0.254mm}\path(195,-30)(220,-80) 
        \special{color rgb 0 0 0}\put(185,-111){\shortstack{$\bar P_3^{N,c}$}} 
        \special{color rgb 0 0 0}\put(70,-111){\shortstack{$I_3^N$}} 
        \special{color rgb 0 0 0}\allinethickness{0.254mm}\special{sh 0.2}\put(305,-80){\ellipse{70}{20}} 
        \special{color rgb 0 0 0}\allinethickness{0.254mm}\special{sh 0.2}\put(420,-80){\ellipse{70}{20}} 
        \special{color rgb 0 0 0}\allinethickness{0.254mm}\special{sh 0.99}\put(280,-80){\ellipse{4}{4}} 
        \special{color rgb 0 0 0}\allinethickness{0.254mm}\special{sh 0.99}\put(305,-30){\ellipse{4}{4}} 
        \special{color rgb 0 0 0}\allinethickness{0.254mm}\special{sh 0.99}\put(330,-80){\ellipse{4}{4}} 
        \special{color rgb 0 0 0}\allinethickness{0.254mm}\special{sh 0.99}\put(395,-80){\ellipse{4}{4}} 
        \special{color rgb 0 0 0}\allinethickness{0.254mm}\special{sh 0.99}\put(420,-30){\ellipse{4}{4}} 
        \special{color rgb 0 0 0}\allinethickness{0.254mm}\special{sh 0.99}\put(445,-80){\ellipse{4}{4}} 
        \special{color rgb 0 0 0}\put(260,-86){\shortstack{$1$}} 
        \special{color rgb 0 0 0}\put(375,-86){\shortstack{$1$}} 
        \special{color rgb 0 0 0}\put(345,-86){\shortstack{$2$}} 
        \special{color rgb 0 0 0}\put(460,-86){\shortstack{$2$}} 
        \special{color rgb 0 0 0}\allinethickness{0.254mm}\path(280,-80)(305,-30) 
        \special{color rgb 0 0 0}\allinethickness{0.254mm}\path(420,-30)(445,-80) 
        \special{color rgb 0 0 0}\put(410,-111){\shortstack{$P_3^N$}} 
        \special{color rgb 0 0 0}\put(295,-111){\shortstack{$\bar P_3^{N,b}$}} 
        \special{color rgb 0 0 0}\allinethickness{0.254mm}\path(395,-80)(420,-30) 
        \special{color rgb 0 0 0} 
\end{picture}
\caption{\label{n3} $N$-flags on 3 vertices}
\end{center}
\end{figure}

\begin{center}
$T_{\FDF}$
\end{center}

$\alpha\in\scr F^1_2[T_{\FDF}]$ is a directed edge in which the tail vertex is labeled by 1. $A$ is the type of size 2 with $E(A)=\{\langle 1,2\rangle\}$, and $B$ is the type of size 2 with $E(B)=\{\langle 2,1\rangle\}$.

\begin{center}
$T_{\Graph}^\ast$
\end{center}

$p_a\in\scr M_1[T_{\Graph}^\ast]$ is the one-vertex model colored by $a\in {\Bbb Z}_3$. $\rho_a [\nu_a] \in\scr M_2[T_{\Graph}^\ast]$ is an edge [none-edge, respectively], in which both ends are colored by $a$, and  $\rho_{a,b} [\nu_{a,b}] \in\scr M_2[T_{\Graph}^\ast]$ is an edge [none-edge, respectively], in which the two ends are colored by $a\neq b\in {\Bbb Z}_3$.

\begin{center}
$T_{\FDF}^\ast$
\end{center}

For $a\in {\Bbb Z}_3$, $(a)$ is the type of size 1 based on $p_a$.
For $a=1,2$, $\alpha_{a,3-a}\in \scr F^{(a)}_2$ is a directed edge in which the tail is colored by $a$ and is labeled by 1, and the head is colored by $3-a$.

$N_a$ is the type of size 2 based on $\nu_a$. In the flag $\vec P_3^{N_a}\in \scr F^{N_a}_3$, the only free vertex $v$ is colored by $(3-a)$, and it has the edges $\langle \theta(1),v\rangle, \langle v,\theta(2)\rangle$. $\vec P_3^{N_a,\ast}\in \scr F^{N_a}_3$ is obtained from $\vec P_3^{N_a}$ by reversing edges.

\begin{center}
{\bf Interpretations $C,OC$}
\end{center}

One of our main tools in this paper is the (total) interpretation $\C:T_{\Graph}^\ast\leadsto T_{\Graph}^E$ that acts identically on $T_{\Graph}$ and, combinatorially, given an edge $E$ of our graph, colors all other vertices in 3 colors according to Figure \ref{c_int} (dotted lines mean the absence of an edge).
\begin{figure}[tb]
\begin{center}
\setlength{\unitlength}{0.254mm}
\begin{picture}(238,345)(115,-370)
        \special{color rgb 0 0 0}\allinethickness{0.254mm}\special{sh 0.2}\put(210,-200){\ellipse{70}{20}} 
        \special{color rgb 0 0 0}\allinethickness{0.254mm}\special{sh 0.99}\put(185,-200){\ellipse{4}{4}} 
        \special{color rgb 0 0 0}\allinethickness{0.254mm}\special{sh 0.99}\put(235,-200){\ellipse{4}{4}} 
        \special{color rgb 0 0 0}\put(160,-206){\shortstack{$c_1$}} 
        \special{color rgb 0 0 0}\put(250,-206){\shortstack{$c_2$}} 
        \special{color rgb 0 0 0}\allinethickness{0.254mm}\path(235,-200)(185,-200) 
        \special{color rgb 0 0 0}\allinethickness{0.254mm}\special{sh 0.1}\put(150,-87){\ellipse{70}{125}} 
        \special{color rgb 0 0 0}\allinethickness{0.254mm}\special{sh 0.99}\put(150,-80){\ellipse{4}{4}} 
        \special{color rgb 0 0 0}\allinethickness{0.254mm}\path(150,-80)(235,-200) 
        \special{color rgb 0 0 0}\allinethickness{0.254mm}\dottedline{5}(150,-80)(185,-200) 
        \special{color rgb 0 0 0}\put(130,-66){\shortstack{\Large $\chi_1$}} 
        \special{color rgb 0 0 0}\allinethickness{0.254mm}\special{sh 0.1}\put(295,-87){\ellipse{70}{125}} 
        \special{color rgb 0 0 0}\put(275,-66){\shortstack{\Large $\chi_2$}} 
        \special{color rgb 0 0 0}\allinethickness{0.254mm}\special{sh 0.99}\put(295,-80){\ellipse{4}{4}} 
        \special{color rgb 0 0 0}\allinethickness{0.254mm}\dottedline{5}(235,-200)(295,-80) 
        \special{color rgb 0 0 0}\allinethickness{0.254mm}\path(295,-80)(185,-200) 
        \special{color rgb 0 0 0}\allinethickness{0.254mm}\special{sh 0.1}\put(210,-307){\ellipse{70}{125}} 
        \special{color rgb 0 0 0}\put(190,-341){\shortstack{\Large $\chi_0$}} 
        \special{color rgb 0 0 0}\allinethickness{0.254mm}\special{sh 0.99}\put(185,-320){\ellipse{4}{4}} 
        \special{color rgb 0 0 0}\allinethickness{0.254mm}\special{sh 0.99}\put(230,-320){\ellipse{4}{4}} 
        \special{color rgb 0 0 0}\allinethickness{0.254mm}\path(185,-200)(185,-320) 
        \special{color rgb 0 0 0}\allinethickness{0.254mm}\path(235,-200)(185,-320) 
        \special{color rgb 0 0 0}\allinethickness{0.254mm}\dottedline{5}(235,-200)(230,-320) 
        \special{color rgb 0 0 0}\allinethickness{0.254mm}\dottedline{5}(185,-200)(230,-320) 
        \special{color rgb 0 0 0} 
\end{picture}
\caption{\label{c_int} Interpretation $C$}
\end{center}
\end{figure}
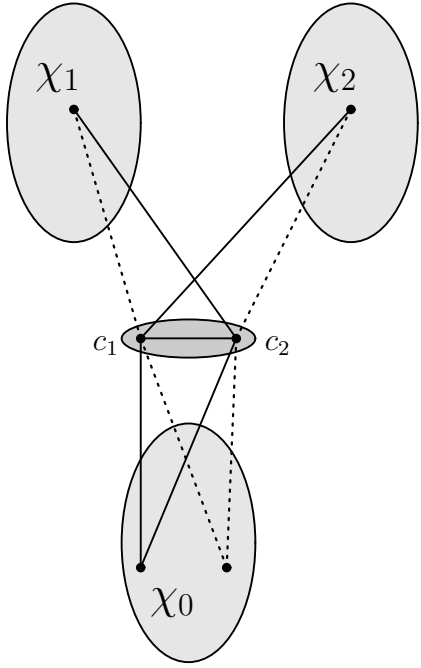
Formally, the predicate symbol $E(v,w)$ of the theory $T_{\Graph}$ is interpreted by itself, and the new symbols $\chi_0,\chi_1,\chi_2$ are interpreted as
\begin{equation} \label{color_introduction}
\longeq{C(\chi_0)(v)&\df& (E(c_1,v)\equiv E(c_2,v) )\\
C(\chi_1)(v) &\df& v\neq c_1\land E(c_2,v)\land \neg E(c_1,v)\\
C(\chi_2)(v) &\df& v\neq c_2\land E(c_1,v)\land \neg E(c_2,v).}
\end{equation}

This construction commutes well with orientation. That is, we have a total interpretation $\text{OC}:T_{\FDF}^\ast\leadsto T_{\FDF}^A$ preserving oriented edges and introducing colors analogously to \refeq{color_introduction} but ignoring orientation on edges (formally, $E(v,w)$ gets replaced by $E(v,w)\lor E(w,v)$). We have the following commutative diagram of interpretations:
\begin{equation} \label{commutative}
\begin{CD}
T_{\Graph}^\ast @>\C>> T_{\Graph}^E\\ @V{\text{OE}^\ast}VV @VV{\text{OE}^A}V\\ T_{\FDF}^\ast @>\text{OC}>> T_{\FDF}^A
\end{CD}
\end{equation}

\subsection{Main results and related conjectures} \label{main_results}

In the language of flag algebras, we have:

\begin{conjecture}[Tur\'an's $(3,4)$-conjecture]
$\rho_3\geq 4/9$.
\end{conjecture}

\begin{conjecture}[Fon-der-Flaass conjecture]
$\pi^{\FDF}(\rho_3)\geq 4/9.$
\end{conjecture}

Let $\mu$ be any probability measure on Borel subsets of ${\Bbb Z}_3\times {\Bbb R}$ that is invariant under the action of $S_3$ on the first coordinate. This naturally (by replacing densities with obvious integrals, cf. \cite{LoSz}) leads to a homomorphism $\phi_\mu\in\Hom^+(\scr A^0[T_\FDF], {\Bbb R})$. $\phi_\mu|_{\FDF}(\rho_3)=4/9$ for any such measure $\mu$, and these $\phi_\mu|_{\FDF}$ constitute the set of all known $\psi\in\Hom^+(\scr A^0[T_\text{Turan}],{\Bbb R})$ with the property $\psi(\rho_3)=4/9$.

\begin{proposition}[\cite{Fon}] \label{FDF}
$\pi^{\FDF}(\rho_3)\geq 3/7$.
\end{proposition}

Our new results are as follows:

\begin{theorem} \label{warm-up}
$\pi^{\FDF}(\rho_3)\geq 7/16$.
\end{theorem}

\begin{theorem} \label{main}
\begin{tenumerate}
\item $\bar P_3=0\Longrightarrow \pi^{\FDF}(\rho_3)\geq 4/9.$

\item $\rho\geq 2/3-\epsilon \Longrightarrow \pi^{\FDF}(\rho_3)\geq 4/9$, where $\epsilon>0$ is an absolute constant. \label{main:b}
\end{tenumerate}
\end{theorem}

\section{Warm-up: proof of Theorem \ref{warm-up}}

In this section we work in the theory $T_{\FDF}$; recall from Section \ref{names} that we also use in this theory all notation introduced in $T_{\Graph}$ via the orientation-erasing interpretation $\text{OE}$.

It turns out that all computations become much more instructive and transparent if we replace critical quantities by the {\em deviations} from their values in (conjectured) extremal examples. In particular, we let
\begin{equation} \label{coordinate_change}
\longeq{\delta &\df& 4/9-\pi^{\FDF}(\rho_3);\\
\delta_\rho &\df& 2/3-\rho;\\
\delta_{K_3} &\df& 2/9-K_3;\\
\delta_\alpha &\df& 1/3-\alpha.}
\end{equation}
Our ultimate goal is to prove $\delta\leq 0$, and our starting point is to re-write $\delta$ in the form
\begin{equation} \label{main_expression}
\delta =\frac 12(\delta_{K_3}-\delta_\rho-\bar P_3)-3\eval{\delta_\alpha^2}{1}
\end{equation}
(for the definition and properties of the averaging operator $\eval{\cdot}{\sigma}$ see \cite[Section 2.2]{flag}).
This is easily checked by expanding all these elements as linear combinations of the seven orgraphs on Figure \ref{m3}.

The first part
\begin{equation} \label{f_def}
f\df \frac 12 (\delta_{K_3}-\delta_\rho-\bar P_3)
\end{equation}
in \refeq{main_expression} actually belongs to (the image of) $\scr A^0[T_{\Graph}]$ and, moreover, if we ignore for the time being the term $\bar P_3$, it corresponds to the classical problem of bounding from below the number of triangles in a graph with a given number of edges. Let us first apply here Goodman's simple bound $K_3\geq \rho(2\rho-1)$ \cite{Goo}. In the tangential coordinates \refeq{coordinate_change} it can be re-written as
\begin{equation} \label{goodman}
\delta_{K_3}\leq \frac 53\delta_\rho-2\delta_\rho^2
\end{equation}
and, thus, since $\eval{\delta_\alpha}{1}=\frac 12\delta_\rho$, from \refeq{main_expression} we get
$$
\delta\leq \frac 13\delta_\rho-\delta_\rho^2-\frac 12\bar P_3-3\eval{\delta_\alpha^2}{1}\leq \frac 13\delta_\rho-\frac 74\delta_\rho^2.
$$
Note that Proposition \ref{FDF} (that is, $\delta\leq \frac 1{63}$) already follows from this. Also, $\delta_\rho\leq 0$ now implies the desired inequality $\delta\leq 0$, so in what follows we can and will assume $\delta_\rho\geq 0$ (i.e., $\rho\leq 2/3$).

The optimal upper bound on $K_3$ in terms of $\rho$ in the range $\rho\in [1/2,2/3]$ was proved in \cite{Fish} and re-proved with different methods in \cite{triangles,Nik}. While it is not very intuitive in terms of $K_3$: $K_3\geq\frac{(1-\sqrt{4-6\rho})(2+\sqrt{4-6\rho})^2}{18}$, it makes much more sense in the tangential coordinates \refeq{coordinate_change}:
$$
\delta_\rho\geq 0\Longrightarrow \delta_{K_3}\leq \delta_\rho+\frac{\sqrt 6}{3}\delta_\rho^{3/2}
$$
(note that the bound is still formally true in the trivial region $\rho\in [0,1/2]$, although it is no longer optimal there). Combining this with \refeq{main_expression}, we get:
\begin{equation} \label{main_1}
\delta\leq \frac{\sqrt 6}{6}\delta_\rho^{3/2}-\bar P_3-3\eval{\delta_\alpha^2}{1}.
\end{equation}
Theorem \ref{warm-up} is now immediate:
$$
\delta\leq \frac{\sqrt 6}{6}\delta_\rho^{3/2}-\bar P_3-3\eval{\delta_\alpha^2}{1}\leq \frac{\sqrt 6}{6}\delta_\rho^{3/2}-\frac 34\delta_\rho^2\leq \frac 1{144}
$$
(the maximum is attained at $\delta_\rho=1/6$), that is, $\pi^{\FDF}(\rho_3)\geq 7/16$.

\section{Proof of Main Result} \label{main:sec}

We begin with \refeq{main_expression} as our starting point, and we first give two examples illustrating that the task of improving Theorem \ref{warm-up} to $\delta\leq 0$ is more elaborate than it may appear.

Note that neither Proposition \ref{FDF} nor our improvement Theorem \ref{warm-up} use $\vec C_4$-freeness. Our first example shows that in order to get any further, it has to be used somehow.

\begin{example}
Consider a random orientation of a complete balanced bipartite graph. This leads to $\phi\in\Hom^+(\scr A^0[T_{\text{Orgraph}}], {\Bbb R})$ with $\phi(\pi^{\FDF}(\rho_3))=7/16$ ($T_{\text{Orgraph}}$ is the theory of arbitrary orgraphs, not necessarily $\vec C_4$-free). This can be checked either by a direct calculation or by observing that for this particular $\phi$ all inequalities that were used in the proof of Theorem \ref{warm-up} are tight.
\end{example}

Thus, we have to use the $\vec C_4$-freeness, and, looking at the proof of Theorem \ref{warm-up}, our only chance is to use this assumption to show that the orientation should be sufficiently far from out-regular, i.e. that $\eval{\delta_\alpha^2}{1}$ must be substantially larger than the trivial Cauchy-Schwarz bound $\eval{\delta_\alpha^2}{1}\geq \frac 14\delta_\rho^2$ used in the proof of Theorem \ref{warm-up}. Our second example demonstrates that even without induced copies of $\vec C_4$, it still can be of order only $\delta_\rho^{3/2}$ (cf. \refeq{main_1}); in particular we can not hope to get away with Goodman's bound \refeq{goodman}.

\begin{example} \label{main:example}
{\bf There exists $\phi\in\Hom^+(\scr A^0[T_{\rm FDF}], {\Bbb R})$ with $\phi(\bar P_3)=0$ and $\phi(f), \phi(\eval{\delta_\alpha^2}{1})=\Theta(\phi(\delta_\rho^{3/2}))$.}

\noindent
We define $\phi$ from a probability measure on ${\Bbb Z}_3\times {\Bbb R}$ (cf. Section \ref{main_results}), only this time $\mu$ is not ${\Bbb Z}_3$-invariant, albeit very close to this. More precisely, let $\mu$ be the Lebesque measure on the set
$$
S\df \{0\}\times [-1/6+\delta_0/2, 1/6-\delta_0/2]\ \ \cup\ \ \{1\}\times [-1/6-\delta_0/2,1/6]\ \ \cup\ \ \{2\}\times [-1/6,1/6+\delta_0/2] ,
$$
and let $\phi\in\Hom^+(\scr A^0[T_{\FDF}],{\Bbb R})$ be the corresponding homomorphism. Then it is easy to calculate that
$$
\phi(\delta_\rho)=\frac 32\delta_0^2
$$
and (see \refeq{f_def})
\begin{equation} \label{fvsdelta0}
\phi(f) =\frac 34\delta_0^3.
\end{equation}

On the other hand, $\phi(\eval{\delta_\alpha^2}{1})$ is the expectation of $(\alpha[(\rn a,\rn x)]-4/9)^2$, where $(\rn a,\rn x)$ is sampled at random from $S$, and $\alpha[(a,x)]$ is relative out-degree of the vertex $(a,x)\in S$. Our set $S$, however, was specifically designed in such a way that $\alpha[(a,x)]$ can be different from 4/9 only for $x\not \in [-1/6+\delta_0/2,1/6-\delta_0/2]$ (that is, on a set of measure $O(\delta_0)$), and the difference itself is also $O(\delta_0)$. Therefore,
$\phi(\eval{\delta_\alpha^2}{1})=\Theta(\delta_0^3)= \Theta(\phi(\delta_\rho))^{3/2}$.
\end{example}

The latter example is pivotal for our proof in the sense that our overall strategy is to show that it is as bad as it gets. Accordingly, the rest of the proof is split into two rather independent part. Firstly, we, arguing in the theory $T_{\Graph}$, prove that $f\geq 0$ implies that the undirected graph resulting from our orgraph by ignoring orientation can be rather well approximated by a complete tripartite graph in which part 0 has density $p_0<1/3$ with $\frac 34\of{\frac 13-p_0}^3\delta_0^3\geq f$ (cf. \refeq{fvsdelta0}). This is done in Section \ref{good_coloring}, and this is where we use the full computational power of Flag Algebras. Then we concentrate on parts 1 and 2 in our tripartite graph and show that, like in Example \ref{main:example}, the induced orgraph must contain $\Omega(\delta_0)$ vertices of relative out-degree significantly smaller than $\delta_0$ that will give us the desired contribution to $\eval{\delta_\alpha^2}{1}$ (Section \ref{low_out_degree}).

\subsection{Finding a good coloring} \label{good_coloring}

In this section we will be mostly working with the theories $T_{\Graph}$, $T_{\Graph}^\ast$. Keeping in line with \refeq{coordinate_change}, we also let
$$
\delta_0=\frac 13-p_0.
$$

Further, define
$$
\kappa\df \sum_{a\in {\Bbb Z}_3} \rho_a +\sum_{a\neq b\in {\Bbb Z_3}}\nu_{a,b};
$$
this element measures the distance between our original (undirected) graph and the complete tripartite graph defined by the coloring $\chi$. We also let
$$
\widetilde\kappa \df \rho_1+\rho_2+\nu_{1,2}
$$
and
\begin{equation} \label{kappaprime}
\kappa'\df \sum_{a\neq b\in {\Bbb Z}_3}\nu_{a,b}.
\end{equation}

Our goal in this section is to prove the following.
\begin{theorem} \label{graph3:thm}
Let $\phi\in\Hom^+(\scr A^0[T_{\rm FDF}],{\Bbb R})$ be such that
\begin{equation} \label{f_assumption}
\phi(f)\geq 0
\end{equation}
and either $\phi(\bar P_3)=0$ or $\phi(\delta_\rho)\leq 10^{-6}$. Then there exists $\phi^\ast\in\Hom^+(\scr A^0[T_{\FDF}^\ast],{\Bbb R})$ with $\phi^\ast|_{\CE}=\phi$ such that $\psi^\ast\df \phi^\ast|_{\text{OE}^\ast}\in \Hom^+(\scr A^0[T_{\Graph}^\ast], {\Bbb R})$ satisfies the following.
\begin{tenumerate}
\item $\delta_0\geq 0$; \label{graph3:a}

\item in the case $\phi(\bar P_3)=0$, $\kappa=\rho_0$ (and, therefore, $\widetilde{\kappa}=\kappa'=0$); \label{graph3:b}

\item $\delta_0^2\leq 6\delta_\rho$; \label{graph3:c}

\item $f+10^{-5}\kappa\leq \frac 34\delta_0^3$. \label{graph3:d}
\end{tenumerate}
\end{theorem}
\begin{proof}
Our coloring will not depend on the orientation (that is, $\psi^\ast$ will be completely determined by $\psi\df \phi|_{\text{OE}}$). Intuitively, we pick an edge $e$ and ${\Bbb Z}_3$-color our graph as shown on Figure \ref{c_int}. We choose $e$ to minimize the density of bad edges $\kappa$. Then we permute colors if necessary (note that the element $\kappa$ is invariant under such permutation) to make sure that 0 is the least frequent color, ensuring property \ref{graph3:a}.

Let us now begin the formal proof. Consider the random homomorphism $\rn{\phi_A}$ rooted at $\phi$, and choose a particular $\phi^A\in S^A(\phi)$ that:
\begin{enumerate}
\item minimizes $\pi^{\text{OE},A}(K_3^E)$ if $\phi(\bar P_3)=0$;

\item minimizes $\pi^{\text{OE},A}\pi^{\C}(\kappa)$ (see the commutative diagram \refeq{commutative}) if $\phi(\delta_\rho)\leq 10^{-6}$.
\end{enumerate}
Let $\psi\df\phi|_{\text{OE}}$ and $\psi^E\df \phi^A|_{\text{OE}^A}\in \Hom^+(\scr A^E[T_{\Graph}], {\Bbb R})$.
Note that in the first case ($\phi(\bar P_3)=0$) the fact $\phi^A\in S^A(\phi)$ readily implies
\begin{equation} \label{missing_copies}
\psi^E(\bar P_3^E) =\psi^E|_i(\bar P_3^{1,b})=\psi^E|_i(\bar P_3^{1,c})=0\ (i=1,2),
\end{equation}
that is, copies of $\bar P_3$ are missing even if they contain one or two distinguished vertices.

In the first case we simply let
$$
\phi^\ast=\phi^A|_{\text{OC}}
$$
(see again \refeq{commutative}). In the second case we additionally compose $\phi^A|_{\text{OC}}$ with an appropriate color permuting automorphism of $\scr A^0[T_{\FDF}^\ast]$ to make sure that the resulting homomorphism satisfies property \ref{graph3:a} in Theorem \ref{graph3:thm}.

Let $\psi^\ast\df \phi^\ast|_{\text{OE}^\ast}\in \Hom^+(\scr A^0[T_{\Graph}^\ast], {\Bbb R})$. We check all four properties claimed in Theorem \ref{graph3:thm} one by one.

\medskip
\noindent
{\bf Property \ref{graph3:a}.} In the second case ($\phi(\delta_\rho)\leq 10^{-6}$) it follows from the construction.

If $\phi(\bar P_3)=0$ then we first note that (by \refeq{missing_copies}) $\psi^\ast(\delta_0)=\psi^E(1/3-K_3^E)$. Now, we have the following calculation in $T_{\Graph}$:
$$
\eval{1/3-K_3^E}{E} =4f+\frac 73\bar P_3+2\eval{(e-2/3)^2}{1}\geq 4f,
$$
which, along with $\phi(f)\geq 0$, implies $\phi(\eval{1/3-K_3^E}{E})\geq 0$ and thus there exists at least one extension $\psi^E\in S^E(\psi)$ such that
\begin{equation} \label{good_e}
\psi^E{K_3^E}\leq 1/3.
\end{equation}
 But by \cite[Theorem 4.1]{flag}, the distribution of $\rn{\psi^E}$ is a convex combination of distributions of $\rn{\phi^A}|_{\text{OE}^A}$ and $\rn{\phi^B}|_{\text{OE}^B}$. Since $K_3^E$ is symmetric w.r.t. permuting labels, all three real-valued random variables  $\rn{\psi^E}{K_3^E}, \rn{\phi^A}|_{\text{OE}^A}(K_3^E)$ and $\rn{\phi^B}|_{\text{OE}^B}(K_3^E)$ have identical distributions. And since our particular choice of $\phi^A$ minimizes $\phi^A|_{\text{OE}^A}(K_3^E)$, property \ref{graph3:a} follows from \refeq{good_e}.

 \medskip
\noindent
{\bf Property \ref{graph3:b}:} by inspecting Figure \ref{c}. $\phi(\bar P_3)=0$ and \refeq{missing_copies} readily imply that the only non-zero contribution to $\psi^\ast(\kappa)$ can be made by copies of $K_4^E$ and, therefore, $\psi^\ast(\kappa)=\psi^\ast(\rho_0)$.

\smallskip
Checking the critical properties \ref{graph3:c} and \ref{graph3:d} is a typical (and rather cumbersome in the second case) calculation in the flag algebras found with the help of a computer. Before we proceed, let us note that even if in the first case $\phi^A|_{\text{OE}^A}$ was chosen to minimize $K_3^E$, it nonetheless minimizes $\phi^A(\kappa)$ (which in this case is the same as $K_4^E$ by property \ref{graph3:b}) as well. Combinatorially this is obvious: $\psi$ corresponds to a complete $k$-partite graph (with $k$ possibly being countably infinite) and both $K_3^E$ and $K_4^E$ are minimized by the same edges connecting the two largest parts. We omit a straightforward formal proof.

\medskip
\noindent
{\bf Property \ref{graph3:c}.} Note that $\psi^\ast(\kappa)=\psi^E(\pi^{\C}(\kappa))$ (since $\kappa$ is invariant under color-permuting automorphisms). On the other hand,
$$
\expect{\rn{\psi^E}(\pi^{\C}(\kappa))} =\frac{\eval{\pi^\C(\kappa)}{E}}{\psi(\rho)}
$$
(\cite[Definition 10]{flag}). Since our chosen $\psi^E\in S^E(\psi)$ minimizes $\pi^{\C}(\kappa)$ (as before, due to the fact that $\pi^{\C}(\kappa)\in\scr F^E_3[T_{\Graph}]$ is symmetric under the permutation of labels), we conclude that
$$
\psi(\eval{\pi^{\C}(\kappa)}{E}) \geq \psi^\ast(\kappa)\psi(\rho).
$$
Lifting $\eval{\pi^{\C}(\kappa)}{E}$ and $\rho$ to $\scr A^0[T_{\Graph}^\ast]$ via the color-erasing interpretation, we can simply say that $\psi^\ast$ obeys the inequality
\begin{equation} \label{optimality}
\pi^{\CE}(\eval{\pi^{\C}(\kappa)}{E}) \geq \kappa\cdot \pi^{\CE}(\rho).
\end{equation}

Let
\begin{equation} \label{f1}
\longeq{f_1&\df& \frac 49(1-9(p_0p_1+p_0p_2+p_1p_2)^2)\\ &=& (\delta_0^2+\frac 13(p_1-p_2)^2)(1+3(p_0p_1+p_0p_2+p_1p_2))\geq \delta_0^2.}
\end{equation}
Let $f_2\in\scr A^0_4[T_{\Graph}]$ be defined as
\begin{equation} \label{f2}
f_2\df 4(\eval{e^2}{1}-\rho^2)+\eval{(P_3^N-2I_3^N)^2}{N}\geq 0.
\end{equation}
Then the inequality in the property \ref{graph3:c} follows from \refeq{f_assumption}, \refeq{optimality} \refeq{f1}, \refeq{f2} and the inequality
\begin{equation} \label{main_inequality}
f_1+2(\pi^{\CE}(\eval{\pi^{\C}(\kappa)}{E}) - \kappa\cdot \pi^{\CE}(\rho)) +20\pi^{\CE}(f)+2\pi^{\CE}(f_2)\leq 6\pi^{\CE}(\delta_\rho).
\end{equation}

\refeq{main_inequality} itself is checked by a direct calculation of coefficients in front of all models from $\scr A^0[T_{\Graph}^\ast]$ (there are 357 of them). It might be helpful to note that there is only one additive term in \refeq{main_inequality} that mixes up the two structures and re-write it as
\begin{equation}\label{inequality}
    \kappa\pi^{\CE}(\rho)\geq \frac 12f_1+\pi^{\CE}(\eval{\pi^{\C}(\kappa)}{E}+10f+f_2-3\delta_\rho);
\end{equation}
note also that both sides here are invariant under permuting colors.

Figure \ref{color} tabulates the values of $\frac 12f_1$ for models in $\scr A^0[T_{\Graph}^\ast]$ where $\chi$ colors vertices according to the scheme shown in this figure ($a,b,c\in {\Bbb Z}_3$ are pairwise different).
\begin{figure}[tb]
\begin{center}
\setlength{\unitlength}{0.254mm}
\begin{picture}(473,87)(20,-96)
        \special{color rgb 0 0 0}\allinethickness{0.254mm}\special{sh 0.99}\put(40,-15){\ellipse{4}{4}} 
        \special{color rgb 0 0 0}\allinethickness{0.254mm}\special{sh 0.99}\put(40,-65){\ellipse{4}{4}} 
        \special{color rgb 0 0 0}\allinethickness{0.254mm}\special{sh 0.99}\put(90,-15){\ellipse{4}{4}} 
        \special{color rgb 0 0 0}\allinethickness{0.254mm}\special{sh 0.99}\put(90,-65){\ellipse{4}{4}} 
        \special{color rgb 0 0 0}\put(20,-21){\shortstack{$a$}} 
        \special{color rgb 0 0 0}\put(20,-71){\shortstack{$a$}} 
        \special{color rgb 0 0 0}\put(100,-71){\shortstack{$a$}} 
        \special{color rgb 0 0 0}\put(100,-21){\shortstack{$a$}} 
        \special{color rgb 0 0 0}\put(55,-96){\shortstack{$2/9$}} 
        \special{color rgb 0 0 0}\allinethickness{0.254mm}\special{sh 0.99}\put(165,-15){\ellipse{4}{4}} 
        \special{color rgb 0 0 0}\allinethickness{0.254mm}\special{sh 0.99}\put(165,-65){\ellipse{4}{4}} 
        \special{color rgb 0 0 0}\allinethickness{0.254mm}\special{sh 0.99}\put(215,-15){\ellipse{4}{4}} 
        \special{color rgb 0 0 0}\allinethickness{0.254mm}\special{sh 0.99}\put(215,-65){\ellipse{4}{4}} 
        \special{color rgb 0 0 0}\put(145,-21){\shortstack{$a$}} 
        \special{color rgb 0 0 0}\put(145,-71){\shortstack{$a$}} 
        \special{color rgb 0 0 0}\put(225,-71){\shortstack{$b$}} 
        \special{color rgb 0 0 0}\put(225,-21){\shortstack{$a$}} 
        \special{color rgb 0 0 0}\put(180,-96){\shortstack{$2/9$}} 
        \special{color rgb 0 0 0}\allinethickness{0.254mm}\special{sh 0.99}\put(290,-15){\ellipse{4}{4}} 
        \special{color rgb 0 0 0}\allinethickness{0.254mm}\special{sh 0.99}\put(290,-65){\ellipse{4}{4}} 
        \special{color rgb 0 0 0}\allinethickness{0.254mm}\special{sh 0.99}\put(340,-15){\ellipse{4}{4}} 
        \special{color rgb 0 0 0}\allinethickness{0.254mm}\special{sh 0.99}\put(340,-65){\ellipse{4}{4}} 
        \special{color rgb 0 0 0}\put(270,-21){\shortstack{$a$}} 
        \special{color rgb 0 0 0}\put(270,-71){\shortstack{$b$}} 
        \special{color rgb 0 0 0}\put(350,-71){\shortstack{$b$}} 
        \special{color rgb 0 0 0}\put(350,-21){\shortstack{$a$}} 
        \special{color rgb 0 0 0}\put(295,-96){\shortstack{$-1/9$}} 
        \special{color rgb 0 0 0}\allinethickness{0.254mm}\special{sh 0.99}\put(415,-15){\ellipse{4}{4}} 
        \special{color rgb 0 0 0}\allinethickness{0.254mm}\special{sh 0.99}\put(415,-65){\ellipse{4}{4}} 
        \special{color rgb 0 0 0}\allinethickness{0.254mm}\special{sh 0.99}\put(465,-15){\ellipse{4}{4}} 
        \special{color rgb 0 0 0}\allinethickness{0.254mm}\special{sh 0.99}\put(465,-65){\ellipse{4}{4}} 
        \special{color rgb 0 0 0}\put(395,-21){\shortstack{$a$}} 
        \special{color rgb 0 0 0}\put(395,-71){\shortstack{$b$}} 
        \special{color rgb 0 0 0}\put(475,-71){\shortstack{$c$}} 
        \special{color rgb 0 0 0}\put(475,-21){\shortstack{$a$}} 
        \special{color rgb 0 0 0}\put(420,-96){\shortstack{$-1/9$}} 
        \special{color rgb 0 0 0} 
\end{picture}
\caption{\label{color} Values of $\frac 12f_1$}
\end{center}
\end{figure}

Likewise, Figure \ref{element} represents the element $\eval{\pi^{\C}(\kappa)}{E}+10f+f_2-3\delta_\rho\in \scr A^0_4[T_\Graph]$.
\begin{figure}[tb]
\begin{center}
\input{element.eepic}
\caption{\label{element} $T_\Graph$-part}
\end{center}
\end{figure}
Then the inequality \refeq{inequality} basically says that if we arbitrarily align the first or the second picture on Figure \ref{color} with the third or fourth picture on Figure \ref{element} (all other cases result in a non-positive coefficient in the right-hand side of \refeq{inequality} and hence are trivial), then for at least two edges their complement will contribute to $\kappa$. This is clear.

\medskip
\noindent
{\bf Property \ref{graph3:d}.}
We form the element
\begin{equation}\label{f3_def}
    \longeq{f_3 &\df& \frac 34\delta_0^3-f+\frac{87}{40}\kappa'\delta_\rho -\frac 1{39240}{\kappa}\\ && \hspace{\parindent}-\frac{43}{160}(\pi^{\CE}(\eval{\pi^{\C}(\kappa)}{E}) -\kappa\cdot \pi^{\CE}(\rho))\\ &\leq& \frac 34\delta_0^3-f+\frac{87}{40}\kappa'\delta_\rho -2\cdot 10^{-5}\kappa.}
\end{equation}
We claim that $f_3\geq 0$; let us first check that is suffices for finishing the proof of property \ref{graph3:d}. We only have to take care of the term $\frac{87}{40}\kappa'\delta_\rho$ in \refeq{f3_def}.

Indeed, if $\phi(\bar P_3)=0$ then $\psi^\ast(\kappa')=0$ by already proven part \ref{graph3:b}, and we are done. On the other hand, if $0\leq \delta_\rho \leq 10^{-6}$ then $\frac{87}{40}\kappa'\delta_\rho\leq \frac{87}{40}\kappa\delta_\rho\leq 10^{-5}\kappa$, and $f_3\geq 0$ still implies property \ref{graph3:d}.

\medskip
Our proof of $f_3\geq 0$ is of distinct brute-force nature, and it is obtained by finding a sufficiently good rational approximation to the numerical outcome of the corresponding semi-definite program. Namely,
\begin{equation}\label{f3_main}
    f_3\geq \delta_0p_0\of{\frac 94(p_1-p_2)^2+\frac{491}{654}\kappa} +\frac 1{5232}\sum_{i=1}^8\eval{Q_i(\vec g_i)}{\sigma_i},
\end{equation}
where $\sigma_1,\ldots,\sigma_8$ are types, $\vec g_i=(g_{i1},\ldots,g_{id_i})$ tuples of elements from $\scr A^{\sigma_i}$, and $Q_i$ are positive semidefinite quadratic forms with integer coefficients given by $(d_i\times d_i)$ positive semidefinite matrices $M_i$.
More precisely, $\sigma_1=\sigma_2\df 0$. $d_1=4$, $\vec g_1\df (\nu_{12},v_{01}+\nu_{02},\rho_1+\rho_2,\rho_0)$, and
$$
M_1\df \left[ \begin {array}{cccc} 6549&5619&-598&-1468\\\noalign{\medskip}
5619&7868&-832&-2094\\\noalign{\medskip}-598&-832&99&150
\\\noalign{\medskip}-1468&-2094&150&1406\end {array} \right].
$$
$d_2=3$, $\vec g_2=(\nu_2-\nu_1-\rho_{02}+\rho_{01}, \nu_{02}-\nu_{01},\rho_2-\rho_{1})$, and
$$
M_2\df  \left[ \begin {array}{ccc} 1308&-598&-209\\\noalign{\medskip}-598&279
&95\\\noalign{\medskip}-209&95&39\end {array} \right].
$$

The remaining types $\sigma_3,\ldots,\sigma_8$ have size 2 and are based on the models $\rho_0,\nu_{12},\rho_1,\nu_{01},\rho_2,\nu_{02}$, respectively. We assume that in the non-symmetric types $\sigma_4,\sigma_6,\sigma_8$ the label 1 is received by the vertex that has smaller ($o<1<2$) color. $d_i=12\ (3\leq i\leq 8)$, and $\vec g_i$ simply enumerates all flags in $\scr F^{\sigma_i}_3$ in a certain order. Let us specify this order, uniformly for all $3\leq i\leq 8$.

For any type $\sigma$ of size 2, every flag $F=(M,\theta)\in \scr F^\sigma_3[T_\Graph^\ast]$ is uniquely determined by a triple $(a,\epsilon_1,\epsilon_2)\ (a\in {\Bbb Z}_3,\epsilon_i\in \{0,1\})$, where $a$ is the color of the unique free vertex $v\in V(M)$, and $\epsilon_i=1$ if and only if $(\theta(i),v)\in E(M)$. We enumerate $\scr F^\sigma_3$ by assigning the number $j=1+4a+2\epsilon_1+\epsilon_2$ to this flag, and, for $\sigma=\sigma_i$, it is the $j$th element in $\vec g_i$.

The matrices $M_i$ are given by
$$
M_3\df \tiny \left[ \begin {array}{cccccccccccc} 837&0&0&999&470&-222&-222&-2&470&
-222&-222&-2\\\noalign{\medskip}0&0&0&0&0&0&0&0&0&0&0&0
\\\noalign{\medskip}0&0&0&0&0&0&0&0&0&0&0&0\\\noalign{\medskip}999&0&0
&1209&565&-267&-267&-2&565&-267&-267&-2\\\noalign{\medskip}470&0&0&565
&1235&1487&1487&1317&-697&-1738&-1738&-1320\\\noalign{\medskip}-222&0&0
&-267&1487&6186&-661&2209&-1738&778&-6063&-2209\\\noalign{\medskip}-
222&0&0&-267&1487&-661&6186&2209&-1738&-6063&778&-2209
\\\noalign{\medskip}-2&0&0&-2&1317&2209&2209&1812&-1320&-2209&-2209&-
1807\\\noalign{\medskip}470&0&0&565&-697&-1738&-1738&-1320&1235&1487&
1487&1317\\\noalign{\medskip}-222&0&0&-267&-1738&778&-6063&-2209&1487&
6186&-661&2209\\\noalign{\medskip}-222&0&0&-267&-1738&-6063&778&-2209&
1487&-661&6186&2209\\\noalign{\medskip}-2&0&0&-2&-1320&-2209&-2209&-
1807&1317&2209&2209&1812\end {array} \right]
$$
$$
M_4\df \tiny \left[ \begin {array}{cccccccccccc} 2916&3470&3470&-4218&5950&6718&
4584&-2394&5950&4584&6718&-2394\\\noalign{\medskip}3470&4660&4128&-
4644&6626&7794&5940&-3322&6464&6048&8854&-2436\\\noalign{\medskip}3470
&4128&4660&-4644&6464&8854&6048&-2436&6626&5940&7794&-3322
\\\noalign{\medskip}-4218&-4644&-4644&9928&-8710&-12116&-6680&3014&-
8710&-6680&-12116&3014\\\noalign{\medskip}5950&6626&6464&-8710&20386&
19018&8784&-3532&7288&7046&5488&-6446\\\noalign{\medskip}6718&7794&
8854&-12116&19018&26982&13098&-2558&5488&10994&10042&-7774
\\\noalign{\medskip}4584&5940&6048&-6680&8784&13098&8742&-3386&7046&
8466&10994&-4004\\\noalign{\medskip}-2394&-3322&-2436&3014&-3532&-2558
&-3386&3196&-6446&-4004&-7774&880\\\noalign{\medskip}5950&6464&6626&-
8710&7288&5488&7046&-6446&20386&8784&19018&-3532\\\noalign{\medskip}
4584&6048&5940&-6680&7046&10994&8466&-4004&8784&8742&13098&-3386
\\\noalign{\medskip}6718&8854&7794&-12116&5488&10042&10994&-7774&19018
&13098&26982&-2558\\\noalign{\medskip}-2394&-2436&-3322&3014&-6446&-
7774&-4004&880&-3532&-3386&-2558&3196\end {array} \right]
$$
$$
M_5\df \tiny \left[ \begin {array}{cccccccccccc} 16258&8232&8232&6430&3311&-1191&-
1191&-3036&5681&-128&-128&-3056\\\noalign{\medskip}8232&10089&7010&
6750&-3325&-651&-1830&-3146&1142&-1394&-2761&-2975\\\noalign{\medskip}
8232&7010&10089&6750&-3325&-1830&-651&-3146&1142&-2761&-1394&-2975
\\\noalign{\medskip}6430&6750&6750&5338&-2682&-980&-980&-2485&862&-
1657&-1657&-2348\\\noalign{\medskip}3311&-3325&-3325&-2682&6397&485&
485&1219&3141&2273&2273&1007\\\noalign{\medskip}-1191&-651&-1830&-980&
485&411&-47&456&-164&564&39&431\\\noalign{\medskip}-1191&-1830&-651&-
980&485&-47&411&456&-164&39&564&431\\\noalign{\medskip}-3036&-3146&-
3146&-2485&1219&456&456&1163&-424&763&763&1094\\\noalign{\medskip}5681
&1142&1142&862&3141&-164&-164&-424&2681&753&753&-503
\\\noalign{\medskip}-128&-1394&-2761&-1657&2273&564&39&763&753&1237&
620&679\\\noalign{\medskip}-128&-2761&-1394&-1657&2273&39&564&763&753&
620&1237&679\\\noalign{\medskip}-3056&-2975&-2975&-2348&1007&431&431&
1094&-503&679&679&1045\end {array} \right]
$$
$$
M_6\df \tiny \left[ \begin {array}{cccccccccccc} 25636&28299&16944&7514&4589&1193&
7024&-8103&3037&450&3967&-11616\\\noalign{\medskip}28299&35709&23685&
8218&7024&3660&14870&-11277&3761&1519&5820&-14960\\\noalign{\medskip}
16944&23685&16839&5342&5547&3036&13047&-8122&2485&945&4210&-10189
\\\noalign{\medskip}7514&8218&5342&16801&-6821&-8516&-3230&-1058&-7132
&-11326&-3312&-1718\\\noalign{\medskip}4589&7024&5547&-6821&17546&4336
&19158&-7091&10937&4364&6238&-7398\\\noalign{\medskip}1193&3660&3036&-
8516&4336&6926&5276&-1747&4273&7886&3242&-2104\\\noalign{\medskip}7024
&14870&13047&-3230&19158&5276&27620&-10496&10353&3134&7497&-10810
\\\noalign{\medskip}-8103&-11277&-8122&-1058&-7091&-1747&-10496&5243&-
3855&-802&-3230&6125\\\noalign{\medskip}3037&3761&2485&-7132&10937&
4273&10353&-3855&7738&5140&4274&-4235\\\noalign{\medskip}450&1519&945&
-11326&4364&7886&3134&-802&5140&9786&3345&-1212\\\noalign{\medskip}
3967&5820&4210&-3312&6238&3242&7497&-3230&4274&3345&2958&-3730
\\\noalign{\medskip}-11616&-14960&-10189&-1718&-7398&-2104&-10810&6125
&-4235&-1212&-3730&7536\end {array} \right]
$$
$$
M_7\df \tiny \left[ \begin {array}{cccccccccccc} 16258&8232&8232&6430&5681&-128&-
128&-3056&3311&-1191&-1191&-3036\\\noalign{\medskip}8232&10089&7010&
6750&1142&-1394&-2761&-2975&-3325&-651&-1830&-3146\\\noalign{\medskip}
8232&7010&10089&6750&1142&-2761&-1394&-2975&-3325&-1830&-651&-3146
\\\noalign{\medskip}6430&6750&6750&5338&862&-1657&-1657&-2348&-2682&-
980&-980&-2485\\\noalign{\medskip}5681&1142&1142&862&2681&753&753&-503
&3141&-164&-164&-424\\\noalign{\medskip}-128&-1394&-2761&-1657&753&
1237&620&679&2273&564&39&763\\\noalign{\medskip}-128&-2761&-1394&-1657
&753&620&1237&679&2273&39&564&763\\\noalign{\medskip}-3056&-2975&-2975
&-2348&-503&679&679&1045&1007&431&431&1094\\\noalign{\medskip}3311&-
3325&-3325&-2682&3141&2273&2273&1007&6397&485&485&1219
\\\noalign{\medskip}-1191&-651&-1830&-980&-164&564&39&431&485&411&-47&
456\\\noalign{\medskip}-1191&-1830&-651&-980&-164&39&564&431&485&-47&
411&456\\\noalign{\medskip}-3036&-3146&-3146&-2485&-424&763&763&1094&
1219&456&456&1163\end {array} \right]
$$
$$
M_8\df \tiny \left[ \begin {array}{cccccccccccc} 25636&28299&16944&7514&3037&450&
3967&-11616&4589&1193&7024&-8103\\\noalign{\medskip}28299&35709&23685&
8218&3761&1519&5820&-14960&7024&3660&14870&-11277\\\noalign{\medskip}
16944&23685&16839&5342&2485&945&4210&-10189&5547&3036&13047&-8122
\\\noalign{\medskip}7514&8218&5342&16801&-7132&-11326&-3312&-1718&-
6821&-8516&-3230&-1058\\\noalign{\medskip}3037&3761&2485&-7132&7738&
5140&4274&-4235&10937&4273&10353&-3855\\\noalign{\medskip}450&1519&945
&-11326&5140&9786&3345&-1212&4364&7886&3134&-802\\\noalign{\medskip}
3967&5820&4210&-3312&4274&3345&2958&-3730&6238&3242&7497&-3230
\\\noalign{\medskip}-11616&-14960&-10189&-1718&-4235&-1212&-3730&7536&
-7398&-2104&-10810&6125\\\noalign{\medskip}4589&7024&5547&-6821&10937&
4364&6238&-7398&17546&4336&19158&-7091\\\noalign{\medskip}1193&3660&
3036&-8516&4273&7886&3242&-2104&4336&6926&5276&-1747
\\\noalign{\medskip}7024&14870&13047&-3230&10353&3134&7497&-10810&
19158&5276&27620&-10496\\\noalign{\medskip}-8103&-11277&-8122&-1058&-
3855&-802&-3230&6125&-7091&-1747&-10496&5243\end {array} \right].
$$
We have little to add here but what was already said before: this is a sufficiently close rational approximation to a numerical solution of the corresponding SDP. \refeq{f3_main} is checked by comparing coefficients in front of all 357 models from $\scr A^0_4[T_\Graph^\ast]$. The only noticeable piece of structure here is the observation that the inequality \refeq{f3_main} is invariant under the automorphism permuting colors 1 and 2. Accordingly, $Q_1$ and $Q_2$ represent positive and negative parts of the corresponding quadratic form (cf. \cite[Section 4]{turan}), and in the pairs $(M_5,M_7)$ and $(M_6,M_8)$ one matrix is obtained from the other by permuting rows and columns according to this automorphism.

Theorem \ref{graph3:thm} is proved.
\end{proof}

\subsection{Finding vertices of low out-degree} \label{low_out_degree}

Now we are ready to finish the proof of Theorem \ref{main}. Fix $\phi\in\Hom^+(\scr A^0[T_\FDF],{\Bbb R})$ such that either $\phi(\bar P_3)=0$ or $\phi(\delta_\rho)\leq 10^{-6}$. We have to show that $\phi(\delta)\leq 0$, where $\delta$ is given by \refeq{main_expression}.

If $\phi(f)\leq 0$, we are done. Otherwise we apply Theorem \ref{graph3:thm} and find an extension $\phi^\ast\in\Hom^+(\scr A^0[T_\FDF^\ast], {\Bbb R})$ such that $\psi^\ast\df \phi^\ast|_{\text{OE}^\ast}$ has all the properties required in its statement.

\smallskip
Before starting the formal argument, let us briefly outline its combinatorial essence. What we have so far, is a $\Bbb Z_3$-coloring of our orgraph such that the color 0 is underrepresented (part \ref{graph3:a} of Theorem \ref{graph3:thm}) by an amount $\delta_0$ that is relatively large with respect to $f$ (part \ref{graph3:d}) and such that the tripartite (unordered) graph defined by this coloring is very close to our original graph after disregarding its orientation.

Our argument focusses on the subgraph induced by colors 1 and 2. Let us imagine for a moment that $\widetilde\kappa=0$ (which is true anyway in the simple case $\phi(\bar P_3)=0$ by part \ref{graph3:b} of Theorem \ref{graph3:thm}). Then this subgraph is an orientation of a {\em complete bipartite} graph that does not contain $\vec C_4$ and hence is acyclic. Therefore, there exists a vertex of out-degree 0 and in the original graph it has (relative) out-degree $\leq p_0=\frac 13-\delta_0$. Proceeding by an obvious induction, we conclude that for any $x$, the fraction of vertices of relative out-degree $\leq \frac 13-\delta_0+x$ is at least $x$, that gives us $\eval{\delta_\alpha^2}{1}\geq \int_0^{\delta_0}(\delta_0-x)^2dx= \frac{\delta_0^3}{3}$. By part \ref{graph3:d} of Theorem \ref{graph3:thm}, this suffices for our purposes.

The main complication is that in general the restriction of our original graph to colors 1 and 2 is not complete bipartite. But, fortunately, by part \ref{graph3:d} of Theorem \ref{graph3:thm}, the difference between them is very small, namely, of order $\delta_0^3$, and we are going to exploit this fact. By Markov's inequality, there are at most $O(\delta_0^{3/2})$ ``bad'' vertices of relative degree $\Omega(\delta_0^{3/2})$ in the difference graph, and we ignore them. This leaves us with the situation when all vertices have a low ($\ll \delta_0$) degree in the difference graph, and the only non-trivial thing we have to show is that our previous argument (when the difference graph is empty) is sufficiently stable to tolerate this imperfection.

\bigskip
Now we begin the formal proof. $\phi^\ast\in \Hom^+(\scr A^0[T_\FDF^\ast],{\Bbb R})$ is fixed throughout the rest of the argument, so we will abbreviate $\phi^\ast(f)\ (f\in A^0[T_\FDF^\ast])$ simply to $f$.

Let
$$
\xi\df \sqrt{\widetilde\kappa},
$$
and assume that $x$ is a real parameter satisfying
\begin{equation} \label{bound_on_x}
7\xi< x\leq \min(p_1,p_2).
\end{equation}
For $a\in \{1,2\}$ we let
$$
Bad_a\df\set{\phi\in\Hom^+(\scr A^{(a)}[T_\FDF^\ast], {\Bbb R})}{\phi(\mu_2^{(a)}(\widetilde\kappa))>\xi}
$$
(recall from \cite[Section 4.3]{flag} that $\mu_2^{(a)}(\widetilde\kappa)$ is the sum of all $\widetilde\kappa$-based flags $F\in \scr F^{(a)}_2[T_\FDF^\ast]$). Let also
$$
Bad_a(x)\df Bad_a \cup \set{\phi\in\Hom^+(\scr A^{(a)}[T_\FDF^\ast],{\Bbb R})}{\phi(\alpha_{a,3-a})<x}.
$$

$p_1,p_2>0$ by \refeq{bound_on_x}, thus we may consider random homomorphisms $\rn{(\phi^\ast)^{(a)}}$ rooted at $\phi^\ast$ that we will abbreviate to $\rn{\phi^a}$. Let
\begin{eqnarray*}
b_a &\df& p_a\cdot\prob{\rn{\phi^a}\in Bad_a};\\
b_a(x) &\df& p_a\cdot\prob{\rn{\phi^a}\in Bad_a(x)}.
\end{eqnarray*}

Note that
$$
\widetilde\kappa =\sum_{a=1}^2 \eval{\mu_2^{(a)}(\widetilde\kappa)}{(a)}
$$
which (by \cite[Definition 10]{flag}) translates to
$$
\widetilde\kappa = \sum_{a=1}^2 p_a\cdot\expect{\rn{\phi^a}(\mu_2^{(a)}(\widetilde\kappa))}.
$$
From this we conclude
$$
\widetilde\kappa \geq \sum_{a=1}^2 p_a\xi\cdot \prob{\rn{\phi^a}(\mu_2^{(a)}(\widetilde\kappa))>\xi} = \xi(b_1+b_2).
$$
Thus,
\begin{equation} \label{upper_on_b}
b_1+b_2\leq\frac{\widetilde\kappa}{\xi}=\xi.
\end{equation}

Our goal is to prove the lower bound
\begin{equation} \label{desired_bound}
b_1(x)+b_2(x)\geq x-6\xi,
\end{equation}
from which the main result will follow relatively easy by an integration on $x$, as explained in our informal description above.

If $b_1(x)=p_1$ or $b_2(x)=p_2$ then the bound trivially follows from the constraint \refeq{bound_on_x}.

Otherwise, closed subsets $S^{(a)}(\phi^\ast)\setminus Bad_a(x)\subseteq \Hom^+(\scr A^{(a)}, {\Bbb R})$ are non-empty, and we fix $\phi^a\in S^{(a)}(\phi^\ast)\setminus Bad_a(x)$ that minimizes $\alpha_{a,3-a}$. Note that by the definition of $Bad_a(x)$, $\phi^a(\alpha_{a,3-a})\geq x$.

Next, we apply \cite[Theorem 4.1]{flag} to the label-removing interpretation $T_\FDF^\ast\leadsto (T_\FDF^\ast)^{(1)}$ described in \cite[Section 2.3.1]{flag} and the homomorphism $\phi^1\in \Hom^+(\scr A^{(1)}[T_\FDF^\ast], {\Bbb R})$. We conclude that the random homomorphism $\rn{\phi^2}$ can be generated by first choosing (under a specific distribution irrelevant to our purposes) a type $\sigma$ of size 2 such that $\sigma|_a=(a)\ (a=1,2)$ and then computing $\rn{(\phi^1)^{\sigma,1}}|_2$. Since $\phi^2\in S^{(2)}(\phi^\sigma)$, and $\Hom^+(\scr A^\sigma, {\Bbb R})$ are compact spaces, this implies the existence of a type $\sigma$ with $\sigma|_a=(a)\ (a=1..2)$ and $\phi^\sigma\in S^{\sigma,1}(\phi^1)$ such that $\phi^\sigma|_2=\phi^2$.

$\phi^\sigma$ and $\phi^a=\phi^\sigma|_a$ will be fixed for the rest of the argument, and we will also drop them from all our equations and inequalities.

Let $F_a=(\Gamma,\theta)\in\scr F^\sigma_3$ be the flag in which the only free vertex $v\in V(\Gamma)$ receives color $a$, and the only new edge is $\langle \theta(3-a),v\rangle$. Then
\begin{equation} \label{bound_on_fa}
F_a\geq \pi^{\sigma,2}(\alpha_{3-a,a})-\pi^{\sigma,1}(\mu_2^a(\widetilde\kappa))\geq x-\xi>0,
\end{equation}
where we used the fact $\phi^a\not\in Bad_a(x)\ (a=1,2)$ and the restriction \refeq{bound_on_x}.

Next, $F_1F_2=\frac 12(G_0+G_1+G_2)$, where the flags $G_a=(\Gamma_a,\theta_a)\in \scr F^\sigma_4$ are defined as follows. $V(\Gamma_a)= \{\theta_a(1),\theta_a(2),v_1,v_2\}$, $G_a|_{\theta_a(1),\theta_a(2),v_b}=F_b\ (b=1..2)$ and the vertices $v_1,v_2$ are independent in $G_0$ and span the edge $\langle v_{3-a},v_a\rangle$ in $G_a$. Note that $G_0\leq \pi^\sigma(\widetilde\kappa)$, thus we have
\begin{equation} \label{bound_on_f1f2}
F_1F_2\leq \frac 12(G_1+G_2+\widetilde\kappa).
\end{equation}
Let $\sigma_a$ be the type of size 3 based on $F_a$ ($\sigma_a|_{[1,2]}=\sigma$, and the free vertex receives label 3). Given that $\phi^\sigma(F_a)>0$ by \refeq{bound_on_fa}, we can form random extensions $\rn{(\phi^\sigma)^{\sigma_a, [1,2]}}$ of $\phi^\sigma$ that we will denote simply by $\rn{\phi^{\sigma_a}}$. Let also $G_a^+\in\scr F^{\sigma_a}_4$ be obtained from $G_a\in\scr F^\sigma_3$ by additionally labeling the vertex $v_a$ with label 3 ($v_{3-a}$ remains unlabeled). Then
$$
\expect{\rn{\phi^{\sigma_a}}(G_a^+)} = \frac 12\frac{G_a}{F_a}.
$$
Substituting this into \refeq{bound_on_f1f2}, we get
$$
F_1F_2\leq \frac 12\widetilde\kappa +\sum_a F_a\expect{\rn{\phi^{\sigma_a}}(G_a^+)}.
$$

Next, we condition according to the event $\rn{\phi^{\sigma_a}}|_3\in Bad_a(x)$ and get the estimate
$$
\expect{\rn{\phi^{\sigma_a}}(G_a^+)}\leq F_{3-a}\cdot\prob{\rn{\phi^{\sigma_a}}|_3\in Bad_a(x)} + \condexpect{\rn{\phi^{\sigma_a}}(G_a^+)}{\rn{\phi^{\sigma_a}}|_3\not\in Bad_a(x)}
$$
(note that since $G_a^+\leq \pi^{\sigma_a,[1,2]}(F_{3-a})$, $\rn{\phi^{\sigma_a}}(G_a^+)\leq F_{3-a}$ with probability 1).
This leads us to
\begin{equation} \label{next_step}
\longeq{F_1F_2 &\leq& \frac 12\widetilde\kappa +\sum_a F_1F_2\prob{\rn{\phi^{\sigma_a}}|_3\in Bad_a(x)}\\&&\hspace{\parindent}+\sum_aF_a\condexpect{\rn{\phi^{\sigma_a}}(G_a^+)}{\rn{\phi^{\sigma_a}}|_3\not\in Bad_a(x)}.}
\end{equation}

Applying once more \cite[Theorem 4.1]{flag} to the label-erasing interpretation $T_\FDF^\ast\leadsto (T_\FDF^\ast)^\sigma$, we can generate $\rn{\phi^a}$ as a convex combination of distributions $\rn{\phi^{\sigma_a'}}|_3$ taken over those $\sigma_a'\in Ext(\sigma,\eta)\ (\eta\function{[2]}{[3]}, \eta(i)=i)$ in which $\chi(3)=a$. In particular, looking at $\sigma_a'=\sigma_a$, we conclude that
$$
F_a \prob{\rn{\phi^{\sigma_a}}|_3\in Bad_a(x)} \leq p_a\prob{\rn{\phi^a}\in Bad_a(x)} = b_a(x).
$$
This takes care of the second term in \refeq{next_step}:
\begin{equation} \label{onemorestep}
F_1F_2\leq \frac 12\widetilde\kappa +\sum_a F_{3-a}b_a(x)+\sum_aF_a\condexpect{\rn{\phi^{\sigma_a}}(G_a^+)}{\rn{\phi^{\sigma_a}}|_3\not\in Bad_a(x)}.
\end{equation}

In order to bound the remaining term $\condexpect{\rn{\phi^{\sigma_a}}(G_a^+)}{\rn{\phi^{\sigma_a}}|_3\not\in Bad_a(x)}$, we fix any particular $\phi^{\sigma_a}\in S^{\sigma_a,[1,2]}(\phi^\sigma)$ with $\phi^{\sigma_a}|_3\not\in Bad_a(x)$. The already made observation based on \cite[Theorem 4.1]{flag} readily implies that $\phi^{\sigma_a}|_3\in S^{(a)}(\phi^\ast)$. Given the way $\phi^a (=\phi^{\sigma_a}|_a)$ was defined, we conclude that $\phi^{\sigma_a}|_3(\alpha_{a,3-a})\geq \phi^{\sigma_a}|_a(\alpha_{a,3-a})$, or, dropping as always the fixed homomorphism $\phi^{\sigma_a}$ from notation,
\begin{equation} \label{two_alpha}
\pi^{\sigma_a,3}(\alpha_{a,3-a}) \geq \pi^{\sigma_a,a}(\alpha_{a,3-a}).
\end{equation}

Next,
\begin{equation} \label{bound_on_graphs}
G_a^+\leq \pi^{\sigma_a,[a,3]}(\vec P_3^{N_a}).
\end{equation}

We now use the fact that our orgraph is $\vec C_4$-free which implies
\begin{equation} \label{p3cop3}
\vec P_3^{N_a}\vec P_3^{N_a,\ast} \leq \frac 12 \pi^{N_a}(\widetilde\kappa).
\end{equation}
This is simply because in every $F=(\Gamma,\theta)\in \scr F^{N_a}_4,\ V(\Gamma)=\{\theta(1),\theta(2),v,w\}$ with $F|_{\{\theta(1),\theta(2),v\}}=\vec P_3^{N_a}$ and $F|_{\{\theta(1),\theta(2),w\}}=\vec P_3^{N_a,\ast}$, $v$ and $w$ can not be independent (otherwise, we would have obtained a copy of $\vec C_4$) and thus the pair $(v,w)$ contributes to $\widetilde\kappa$.

Now, $\vec P_3^{N_a}+\vec K_{2,1}^{N_a}\leq \pi^{N_a,1}(\alpha_{a,3-a})$. Similarly, $\vec P_3^{N_a,\ast}+\vec K_{2,1}^{N_a}\geq \pi^{N_a,2}(\alpha_{a,3-a})-\pi^{N_a,1}(\mu_2^{(a)}(\widetilde\kappa))$. Substracting,
$$
\vec P_3^{N_a,\ast}-\vec P_3^{N_a}\geq \pi^{N_a,2}(\alpha_{a,3-a})-\pi^{N_a,1}(\alpha_{a,3-a})-\pi^{N_a,1}(\mu_2^{(a)}(\widetilde\kappa)).
$$
We lift this inequality to the algebra $\scr A^{\sigma_a}$ via $\pi^{\sigma_a,[a,3]}$. Using also \refeq{two_alpha} and the fact $\phi^a\not\in Bad_a$, we see that
$$
\pi^{\sigma_a,[a,3]}(\vec P_3^{N_a,\ast}) \geq \pi^{\sigma_a,[a,3]}(\vec P_3^{N_a})-\xi.
$$
Lifting to this algebra also the inequality \refeq{p3cop3}, we conclude that $\pi^{\sigma_a,[a,3]}(\vec P_3^{N_a})\cdot (\pi^{\sigma_a,[a,3]}(\vec P_3^{N_a})-\xi)\leq\frac 12\widetilde\kappa$, from which we find $\pi^{\sigma_a,[a,3]}(\vec P_3^{N_a})\leq \xi+\sqrt{\widetilde\kappa/2}\leq 2\xi$. Comparing this inequality with \refeq{bound_on_graphs}, we finally find that for any $\phi^{\sigma_a}\in S^{\sigma_a,[1,2]}(\phi^\sigma)$ with $\phi^{\sigma_a}|_3\not\in Bad_a(x)$, we have $\phi^{\sigma_a}(G_a^+)\leq 2\xi$.

Thus, the bound \refeq{onemorestep} implies
$$
F_1F_2\leq \frac 12\widetilde\kappa +\sum_a F_{3-a}b_a(x) +2\xi\sum_a F_a=\frac 12\widetilde\kappa+\sum_a F_{3-a}(b_a(x)+2\xi).
$$
We can assume w.l.o.g. that $F_1\geq F_2$. Then we further have $F_1F_2\leq \frac 12\widetilde\kappa+F_1(b_1(x)+b_2(x)+4\xi)$. Dividing by $F_1$ and taking into account \refeq{bound_on_fa}, we complete the proof of \refeq{desired_bound} in the non-trivial case $b_1(x)<p_1,\ b_2(x)<p_2$.

\smallskip
Comparing \refeq{desired_bound} with \refeq{upper_on_b}, and consulting the definitions of $Bad_a$ and $Bad_a(x)$, we see that
\begin{equation} \label{preprob}
\sum_a p_a\prob{\rn{\phi^a}(\alpha_{a,3-a})<x \land \rn{\phi^a}(\mu_2^{(a)}(\widetilde\kappa))\leq\xi}\geq x-7\xi.
\end{equation}
On the other hand, applying once more \cite[Theorem 4.1]{flag} to the color-erasing interpretation, we observe that the distribution of $\rn{\phi^1}(\alpha)$ is the convex combination of distributions $\rn{\phi^a}\pi^{\CE,(a)}(\alpha)\ (a\in {\Bbb Z}_3)$ with weights $p_a$. Also,
$$
\pi^{\CE,(a)}(\alpha) \leq \alpha_{a,3-a} + \mu_2^{(a)}(\widetilde\kappa) +p_0= \frac 13 - (\delta_0 -\alpha_{a,3-a}-\mu^{(a)}_2(\widetilde\kappa)).
$$
Thus, \refeq{preprob} implies that for every $x\in (7\xi,\min(p_1,p_2)]$ we have the estimate
$$
\prob{1/3-\rn{\phi^1}(\alpha)>\delta_0-x-\xi} \geq x-7\xi.
$$
Integrating this inequality from $7\xi$ to $\min(p_1,p_2,\delta_0-\xi)$,
\begin{equation} \label{integral}
\longeq{\eval{\delta_\alpha^2}{1} &=& \expect{(1/3-\rn{\phi^1(\alpha)})^2} \geq \int_{7\xi}^{\min(p_1,p_2,\delta_0-\xi)} (\delta_0-x-\xi)^2dx\\ &=& \frac 13((\delta_0-8\xi)^3-\max(0,\delta_0-\xi-p_1,\delta_0-\xi-p_2)^3)\\ &\geq& \frac 13((\delta_0-8\xi)^3-\max(0,p_2-2/3,p_1-2/3)^3).}
\end{equation}

Now we are only left to take care of the case when either $p_1$ or $p_2$ are abnormally large. This requires one more simple calculation in $\scr A^0_3[T_\Graph^\ast]$:
$$
(p_a-2/3)\of{\frac 13+\frac 34\rho_0+\frac 32p_0p_{3-a}} +\pi^{\CE}(f)+(p_a-2/3)^2\leq \frac 12(\kappa-\rho_0).
$$
Since $\psi(f)\geq 0$ by our assumption, we get from here $(p_a-2/3)\leq\frac 32(\kappa-\rho_0)$, and, substituting this into \refeq{integral}, we finally conclude that
\begin{equation} \label{final_inequality}
\eval{\delta_\alpha^2}{1} \geq \frac 13\of{(\delta_0-8\xi)^3-\frac{27}{4}(\kappa-\rho_0)^3}.
\end{equation}

In the simple case $\psi(\bar P_3)=0$, part \ref{graph3:b} of Theorem \ref{graph3:thm} implies $\xi=\widetilde\kappa=\kappa-\rho_0=0$ and hence $\eval{\delta_\alpha^2}{1}\geq \frac 13\delta_0^3$. Now $\delta\leq 0$ follows from the bound $f\leq \frac 34\delta_0^3$ provided by part \ref{graph3:d} of Theorem \ref{graph3:thm}.

On the other hand, when $\delta_\rho$ is arbitrarily small, part \ref{graph3:c} of Theorem \ref{graph3:thm} implies that $\delta_0$ is also arbitrarily small, and then part \ref{graph3:d} implies $\widetilde\kappa\leq\kappa\leq O(\delta_0^3)$. Thus, \refeq{final_inequality} gives us $\eval{\delta_\alpha^2}{1}\geq \delta_0^3\of{\frac 13-o(1)}$, and, again, $\delta\leq 0$ follows (for sufficiently small $\delta_\rho$) from $f\leq \frac 34\delta_0^3$.

Theorem \ref{main} is proved.

\section{Conclusion} \label{concl}

We have proved Tur\'an's conjecture for a natural class of 3-graphs that contain all Tur\'an-Brown-Kostochka examples. This opens up a principal (but, admittedly, somewhat distant at the moment) possibility to attack the general case by trying to construct an inverse interpretation {\em of} the Fon-der-Flaass theory {\em in} the theory of Tur\'an 3-graphs, possibly in some loose sense. It is still too early to tell, however, how promising is this approach.

A more accessible goal might be to remove extra assumptions from Theorem \ref{main}. Most likely, that should entail a significant simplification of our proof that, in our opinion, would be interesting in its own right.

\bibliographystyle{alpha}
\bibliography{razb}
\end{document}